\documentclass[12pt,a4paper]{amsart}
\usepackage{ulem}
\usepackage{fancyhdr,mathtools}
\usepackage{appendix}
\usepackage{amssymb,amscd,amsxtra,calc}
\usepackage{mathrsfs}
\usepackage{multirow}
\usepackage[all]{xy}
\usepackage{longtable}
\usepackage[colorlinks,linkcolor=blue,anchorcolor=red,citecolor=green,linktocpage]{hyperref}
\usepackage{cite}	
\usepackage{mathabx,epsfig}
\def\acts{\ \rotatebox[origin=c]{-90}{$\circlearrowright$}\ }
\def\racts{\ \rotatebox[origin=c]{90}{$\circlearrowleft$}\ }

\theoremstyle{plain}
    \newtheorem{thm}{Theorem}[section]
    
       \newtheorem{lemma}[thm]{Lemma}
           \newtheorem{theorem}[thm]{Theorem}
            
    \newtheorem{proposition}[thm]{Proposition}
\newtheorem{proposition-definition}[thm]{Proposition-Definition}

\theoremstyle{definition}

    \newtheorem{definition}[thm]{Definition}
        \newtheorem{example}[thm]{Example}
     \newtheorem{corollary}[thm]{Corollary}
     
     \newtheorem{prob}[thm]{Problem}
       \newtheorem{conjecture}[thm]{Conjecture}

    \newtheorem*{notation*}{Notation and Terminology}
    \newtheorem{remark}[thm]{Remark}
    
\theoremstyle{remark}

\newcommand{\codim}{\textup{codim}}

\newcommand{\sing}{\mathrm{sing}}
\newcommand{\pr}{\textup{pr}}

\newcommand{\id}{\textup{id}}

\newcommand{\Sym}{\textup{Sym}}

\newcommand{\reg}{\mathrm{reg}}

\begin{document}

\title[Endomorphism]
{Positivity of compact K\"ahler varieties \\ admitting an int-amplified endomorphism}

\author{Shin-ichi Matsumura and Guolei Zhong}

\address{
\textsc{Mathematical Institute 
$\&$ Division for the Establishment of Frontier Science of Organization for Advanced Studies, 
Tohoku University}\endgraf
	\textsc{6-3, Aramaki Aza-Aoba, Aoba-ku, Sendai 980-8578, Japan
}}
\email{mshinichi-math@tohoku.ac.jp}
\email{mshinichi0@gmail.com}

\address{
\textsc{School of Mathematical Sciences, Shanghai Key Laboratory of PMMP}\endgraf 
\textsc{East China Normal University, Shanghai 200241, China}
}
\email{glzhong@math.ecnu.edu.cn}

\thanks{S.\,M.\,was supported by the JST FOREST Program ($\sharp$PMJFR2368).
G.\,Z.\,was supported by the Science and Technology Commission of Shanghai Municipality (No. 22DZ2229014), and the National Natural Science Foundation of China. }

\renewcommand{\subjclassname}{
\textup{2020} Mathematics Subject Classification}


\begin{abstract}
We study compact K\"ahler varieties admitting int-amplified endomorphisms
from the viewpoint of positivity of tangent sheaves.
Our main result shows that the tangent sheaf of such a variety is weakly
positively curved; in particular, it is pseudo-effective in a strong sense.
This establishes a new link between complex dynamics and the positivity theory
of tangent sheaves, and provides an alternative to equivariant MMP techniques
in the compact K\"ahler setting.

As an application, we prove a K\"ahler analogue of Yoshikawa's structure
theorem: for a compact K\"ahler klt variety admitting an int-amplified endomorphism, after an equivariant quasi-\'etale cover, it admits an equivariant flat MRC fibration with irreducible fibres onto a complex torus; in the smooth case, the fibration is smooth whose periodic fibres are of Fano type. 
To this end, we prove the existence of minimal models in the case of numerical
dimension zero canonical divisors for non-projective K\"ahler varieties. 
We further study rationally connected manifolds with pseudo-effective tangent
bundle, focusing on their relation to Fano-type properties, almost
homogeneity, and fundamental groups.
In the proof, we also show that every surjective
endomorphism of a compact klt K\"ahler variety lifts to a suitable maximally
quasi-\'etale cover.
\end{abstract}

\maketitle

\section{Introduction}\label{s:intro}

\subsection{Positivity of tangent sheaves and int-amplified endomorphisms}\label{intro-1}

A central problem in complex dynamics is to understand the geometry of compact
complex varieties admitting non-isomorphic surjective endomorphisms.  In the
projective setting,  remarkable progress has been made in the study of int-amplified
endomorphisms.  
One of the main reasons is that the MMP is well-established and can even 
be run equivariantly with respect to the endomorphism, leading to a precise
structure theory (see \cite{BCHM10} and \cite{MZ23-survey}).  
However, the compact K\"ahler setting presents a substantially different picture,
since the K\"ahler MMP is not yet available in the general case.  

The purpose of this paper is to propose an alternative
mechanism for studying int-amplified endomorphisms in the compact K\"ahler
setting, one which does not rely on the equivariant MMP.  The guiding idea is
that the existence of an int-amplified endomorphism should force positivity of
the tangent sheaf.  This provides a new perspective on the study of dynamical
systems, linking complex dynamics with the positivity theory of tangent
sheaves.

The \textit{dynamical degrees}, extensively studied by Dinh and Sibony, are powerful tools for the cohomological characterization of dominant meromorphic maps (in particular, endomorphisms). 
Let \(f\colon X\dashrightarrow X\) be a dominant meromorphic  self-map 
of a compact K\"ahler manifold $X$ with a fixed K\"ahler form \(\omega\). 
For each integer \(0\leq p\leq n\coloneqq\dim(X)\), the \textit{\(p\)-th dynamical degree} of \(f\) is defined by
\[
d_p(f)\coloneqq\lim_{m\to\infty}\|(f^m)^*(\omega^p)\|^{1/m}
=\lim_{m\to\infty}\|(f^m)^*\colon H^{p,p}(X,\mathbb{C})\to H^{p,p}(X,\mathbb{C})\|^{1/m}.
\]
Here, \(\|T\|\coloneqq\left<T,\omega^{n-p}\right>\) denotes the mass 
of a positive closed \((p,p)\)-current \(T\) on $X$. 
The above limits exist,  are independent of the choice of \(\omega\), and are bimeromorphic invariants; see \cite{DS04-Ens,DS05}.
By the log-concavity of dynamical degrees, the case  \(d_n(f)>d_p(f)\) for all \(p\leq n-1\) is one of the elementary situations in complex dynamics.
For surjective endomorphisms, this condition is also equivalent to $f$ being  \textit{int-amplified}, 
that is, there exists a K\"ahler class \([\omega]\) such that \(f^*[\omega]-[\omega]\) is again K\"ahler;
see \cite[Proposition 3.7]{MZ23-survey} and Proposition \ref{p:int}.

Suppose first that $X$ is a mildly singular projective variety admitting an int-amplified endomorphism \(f\). 
Meng proved in \cite[Theorem 1.7]{Men20} that $X$ admits an everywhere-defined MRC fibration $X \to Y$ 
onto a \(Q\)-abelian variety $Y$, 
that is, a quasi-\'etale quotient of an abelian variety. 
Furthermore, according to \cite[Theorem 1.10]{Men20}, there exists a sequence of divisorial contractions, flips, and Fano contractions 
$$\xymatrix{
        f \acts X\coloneqq X_0 \ar@{-->}[r]&X_1\racts f_1\ar@{-->}[r]&X_2\racts f_2\ar@{-->}[r]&\cdots\ar@{-->}[r]&X_N\racts f_N.
        }$$
such that, after iterates, each $X_{i}$ admits an int-amplified endomorphism \(f_i\) descending from \(f\) and 
the final outcome $X_{N}$ is a \(Q\)-abelian variety. 
As we discussed in the beginning, based on this equivariant MMP, int-amplified endomorphisms have been studied extensively
in both algebraic and arithmetic dynamics. 
However, the situation is markedly different in the compact K\"ahler setting, where the K\"ahler MMP is still not available in the generality needed for such arguments.

On the other hand, a parallel phenomenon has been established for projective manifolds 
with \textit{pseudo-effective} tangent bundle in \cite{HIM22, Mat23}. 
This naturally leads to the following question:
\textit{is there a relation between int-amplified endomorphisms and positivity of tangent sheaves?}

Our first main result gives an affirmative answer: we show
that an int-amplified endomorphism forces a strong form of pseudo-effectivity
of the tangent sheaf. This provides a new bridge between complex dynamics
and the positivity theory of tangent sheaves, and gives a substitute for
equivariant MMP techniques in several questions in compact K\"ahler geometry:
\begin{center}
\textit{Endomorphisms $\Longrightarrow$
positivity of tangent sheaves $\Longrightarrow$ structure theorems.}
\end{center}

\begin{theorem}[{cf.~Theorem \ref{t:pos-refine}}]\label{t:pos}
Let $X$ be a normal compact K\"ahler variety admitting an int-amplified endomorphism. 
Then the tangent sheaf \(T_X\coloneqq\Omega_X^\vee\) is weakly positively curved. 
In particular, \(T_X\) is pseudo-effective in the sense of Definition \ref{df-pos}.
\end{theorem}
Theorem \ref{t:pos} gives a genuinely new positivity result, 
even in some of the most basic cases like 
Fano manifolds of Picard number one. 
Besides, it can also be used as a substitute for the equivariant MMP in the compact K\"ahler setting. 

The pseudo-effectivity used in this paper 
is stronger than that of the tautological line bundle 
\(\mathcal{O}_{\mathbb{P}(T_X)}(1)\) on the associated Grothendieck projectivization; 
see Section \ref{s:ex}.
Moreover, the determinant of a pseudo-effective sheaf in this sense is again pseudo-effective; see \cite[Proposition 3.4]{Mat23}. 
Thus, Theorem \ref{t:pos} 
strengthens \cite[Theorem C]{BdFF12}, 
\cite[Theorem 1.3]{Zho21-Asian} and \cite[Theorem 1.5]{Men20} on the pseudo-effectivity of the anticanonical divisor \(-K_X\) (cf.~\cite[Theorem 6.2]{MZ22} for the non-vanishing results and Conjecture \ref{Conj-big}).

It is natural to ask whether other positivity properties, 
such as singular positive curvature or almost nefness of the tangent sheaf, 
also hold under the assumptions of Theorem \ref{t:pos}. 
However, there are counterexamples to these expectations; see Remark \ref{r:pos-cur}. 
This suggests that pseudo-effectivity is a natural and suitable positivity notion in this setting.

\subsection{Applications of Theorem \ref{t:pos}}\label{intro-2}
We next discuss several applications of Theorem \ref{t:pos} (see Theorem \ref{t:structure} and Corollaries \ref{c:simply} and \ref{Cor:linear}). 
We also refer the reader to Proposition \ref{p:Q-torus} for a dynamical characterization of \(Q\)-complex torus by applying Theorem \ref{t:structure}. 

The first application is the  K\"ahler analogue of Yoshikawa's
structure theorem.

\begin{theorem}\label{t:structure}
Let \(f\colon X\to X\) be an int-amplified endomorphism of a compact klt K\"ahler variety \(X\).
Then, up to an \(f\)-equivariant quasi-\'etale cover, there exists a flat holomorphic projective MRC fibration \(\pi\colon X\to T\) onto a complex torus \(T\), with every fibre irreducible such that \(\dim T=\widehat{q}(X)\), the augmented irregularity.
If  \(X\) is smooth, then \(\pi\)  is a smooth 
fibration and the \(f\)-periodic fibres  are of Fano type.
\end{theorem}

Theorem \ref{t:structure} extends Yoshikawa's result \cite[Theorem 1.3]{Yos21} 
(cf.~Theorem \ref{t:qe-mrc}) from projective varieties to compact K\"ahler varieties. 
Yoshikawa's proof is based on the equivariant MMP, while our proof, by contrast, proceeds through the positivity of
the tangent sheaf obtained in Theorem \ref{t:pos}. 

\begin{remark}[Strategy towards the proof of Theorem \ref{t:structure}]
First, as a key birational inputs,  motivated by a recent breakthrough of the K\"ahler MMP \cite{HX26}, we prove the existence of minimal models for compact K\"ahler klt varieties with numerical dimension zero canonical divisor (see Theorem \ref{t:nd0}), extending \cite{Gon11}.
In this way, we give a solution to \cite[Conjecture 1.7]{MZ-TJM} and renders the structure theorem for pseudo-effective tangent sheaves \cite[Theorem 1.6]{MZ-TJM} unconditionally (see Theorem \ref{t:psef-strc}; cf.~\cite{MZ-TJM}).   
Combining that with Theorem \ref{t:pos} yields a holomorphic MRC fibration onto a complex torus; the dynamics then provide equivariance (see Theorem \ref{thm-lift}),  irreducibility of every fibre, and flatness.
\end{remark}

In singular klt case, Theorem   \ref{t:structure} does not assert that periodic fibres are of Fano type. 
Even when \(X\) is maximally quasi-\'etale, it is unknown whether a general fibre is itself maximally quasi-\'etale. Nevertheless, the following corollary is a direct consequence of Theorem  \ref{t:structure}.
It relates the geometry of a compact klt K\"ahler variety admitting an int-amplified endomorphism to its fundamental group.

\begin{corollary}\label{c:simply}
Let \(f\colon X\to X\) be an int-amplified endomorphism of a compact klt K\"ahler variety.
Then the following statements are equivalent.
\begin{enumerate}
\item The variety \(X\) is of Fano type. 
\item The variety \(X\) has vanishing augmented irregularity.
\item There is an \(f\)-equivariant maximally quasi-\'etale cover \(\widehat{X}\to X\) with \(\pi_1(\widehat{X})\) finite. 
\end{enumerate}
\end{corollary}

For the organization of the paper, we will postpone the proof of Theorem \ref{t:structure} and Corollary \ref{c:simply} near the end of the paper (see Section \ref{s:num0}).

The equality \(\dim T=\widehat{q}(X)\) in Theorem \ref{t:structure} singles out the rationally connected variety as the fundamental building block.
Initiating from Corollary \ref{c:simply}, we have following extension of \cite[Corollary 6.3]{IMZ23}, which is independent of dynamics. 

\begin{theorem}\label{t:psef-rc}
Let \(X\) be a normal compact klt K\"ahler variety whose tangent sheaf is pseudo-effective. 
Then the following assertions are equivalent.
\begin{enumerate}
    \item Every quasi-\'etale cover of \(X\) is rationally connected. 
    \item \(X\) is projective, and the tangent sheaf \(T_X\) is generically ample.
    \item The augmented irregularity \(\widehat{q}(X)\) vanishes.
\end{enumerate}
Moreover, if \(X\) is maximally quasi-\'etale $($see Subsection \ref{sec-qe}$)$, then the above equivalent conditions are also equivalent to the following condition:
\begin{enumerate}
    \item[(4)] \(X\) is simply connected.
\end{enumerate}
\end{theorem}

Regarding the equivalences in Theorem \ref{t:psef-rc}, the maximally quasi-\'etale assumption in Corollary \ref{c:simply}(3) cannot be removed; see Example \ref{ex-mqe}. 

For a surjective endomorphism \(f\colon X \to X\) of a compact K\"ahler variety, lifting or descending \(f\) to suitable canonical models is an  important classification tool.   
For instance, when \(X\) is a projective klt variety, after iterates, one can lift \(f\) to a small \(\mathbb{Q}\)-factorial model (see \cite[Theorem 1.2]{MYY24-lcy} and \cite[Appendix]{CZ25}). 

Another key ingredient of the proof of Theorem \ref{t:structure}  
concerns the
equivariance of maximally quasi-\'etale covers.
The equivariance of suitable quasi-\'etale covers plays a crucial  role in separating \(Q\)-abelian varieties and rationally connected varieties appearing in this structure theory. 
Several previous works, including \cite{NZ10,Men20,Zho21-Asian}, developed the so-called \textit{Albanese closure} to lift endomorphisms from \(Q\)-abelian varieties to their finite covers.
Yoshikawa \cite{Yos21} proved a lifting result for int-amplified endomorphisms to certain quasi-\'etale covers of projective klt varieties.  

Our next result combines and extends these results by showing that there exists a suitable maximally quasi-\'etale cover of \(X\) to which \(f\) lifts, even without taking iterates.

\begin{theorem}\label{thm-lift}
Let \(f \colon X \to X\) be a surjective endomorphism of a compact klt K\"ahler variety \(X\).
Then there exists a maximally quasi-\'etale cover
\(\widehat X \to X\) to which \(f\) lifts.
\end{theorem}

As another application of Theorem \ref{t:pos}, together with the fact that smooth hypersurfaces of dimension at least \(3\) and degree at least \(3\) have non-pseudo-effective tangent bundles (see \cite[Theorem 1.4]{HLS22}), we give an alternative proof of \cite{Bea01} 
(cf.~\cite[Proposition 8]{PS89}, \cite[Theorem A]{KT24}). 

\begin{corollary}\label{Cor:linear}
Let \(X\subseteq\mathbb{P}^n\) be a smooth hypersurface of dimension at least \(3\).
Assume that \(X\) admits a non-isomorphic surjective endomorphism.
Then \(X\) is a hyperplane. 
\end{corollary}

It is worth noting that Corollary \ref{Cor:linear} 
fails without the smoothness assumption; indeed, for each \(n\geq 4\), there exists a degree \(n\) klt hypersurface in \(\mathbb{P}^n\) that admits a non-isomorphic surjective endomorphism (see \cite[Example 1.9]{Zha14}).

It is a folklore conjecture that a Fano manifold of Picard number one admitting a non-isomorphic surjective endomorphism is isomorphic to projective space. 
The conjecture was partially solved in \cite{SZ25} under the assumption that \(\mathcal{O}_{\mathbb{P}(T_X)}(1)\) is big.   
Building on this result, Theorem \ref{t:pos} may be viewed as a first step toward the following expectation.

\begin{conjecture}[{cf.~Remark \ref{r:pos-cur} (3)}]\label{Conj-big}
Let \(X\) be a Fano manifold of Picard number one admitting a non-isomorphic surjective endomorphism.
Then there is a positive integer \(m\) such that \(H^0(X,\Sym^mT_X)\neq 0\). 
More optimistically, the tangent bundle \(T_X\) is big.
\end{conjecture}

\subsection{Further questions and examples}\label{intro-3}

We finally discuss examples showing the limitations of pseudo-effectivity of
the tangent bundle when we only consider  the building block--rationally connected varieties--with positive tangent sheaves.

Let $X$ be a
rationally connected projective manifold.
If the tangent bundle $T_X$ is nef, then $X$ is Fano by \cite{DPS94} 
and is expected to be rational homogeneous by the Campana--Peternell conjecture \cite{CP91}. 
On the other hand, if $X$ admits an int-amplified endomorphism, then $X$ is of Fano type by \cite{Yos21} and is even conjecturally to be toric, see \cite[Question 4.4]{Fak03}. 
Therefore, we might expect pseudo-effectivity of the tangent bundle to impose
a similar, although weaker, form of rigidity.

In what follows, we provide examples, showing that 
this expectation is too optimistic: 

\begin{theorem}[{see Examples \ref{e:dp} and  \ref{e:non-ft}}] \label{thm-example}
A rationally connected projective manifold with pseudo-effective tangent bundle 
is not necessarily of Fano type, nor is it necessarily almost homogeneous. 
\end{theorem}

This leads to the following revised problem for rationally connected manifolds.

\begin{prob}\label{prob-smooth}
Let $X$ be a rationally connected projective manifold with pseudo-effective tangent bundle.
Then we ask the following questions.  
\begin{enumerate}
\item[(1)] Is the anticanonical divisor $-K_{X}$ big?
\item[(2)] Does $X$ admit a birational morphism $X \to Y$ such that $Y$ is Fano, or at least of Fano type?
\end{enumerate}
\end{prob}

This paper is organized as follows.
After reviewing the relevant notions and terminology in Section \ref{s:notion}, 
we prove Theorem \ref{t:pos} in Section \ref{s:proof}. 
We then prove Theorem \ref{t:psef-rc} in Section \ref{s:RC}.
In Section \ref{s:rmk}, we prove Theorem \ref{thm-lift}, and discuss Yoshikawa's cover from the viewpoint of maximally quasi-\'etale covers 
and the fundamental groups of regular loci.
In Section \ref{s:num0}, we prove the existence of minimal models for numerical dimension zero canonical divisor in the K\"ahler setting, and then prove Theorem \ref{t:structure}  and Corollary \ref{c:simply}. 
Finally, we present related examples and prove Theorem \ref{thm-example} in Section \ref{s:ex}.

\subsection*{Acknowledgements}
The authors are grateful to Prof.\,Masataka Iwai 
for fruitful discussions concerning Theorem \ref{t:psef-rc}, Example \ref{ex-psef}, and the work \cite{CDM26}. 
The first author is also grateful to Prof.\,Shou Yoshikawa 
for many discussions over the past several years on Theorem \ref{t:structure}. 
The second author would also like to thank Prof. Jie Liu for the useful discussion on Example \ref{e:non-ft}.

\section{Notation and terminology}\label{s:notion}
In this section, we briefly recall the notation and terminology used in the proofs.
We refer the reader to \cite[Subsection 2.2]{IJZ25} 
for a summary of standard notions and terminology concerning compact K\"ahler varieties, 
and to \cite[Section 2]{IMZ23} and the references therein 
for positivity properties of torsion-free sheaves.

\subsection{Singular Hermitian metrics}\label{s:metric}
In this subsection, we follow \cite{PT18,HPS18} and recall the notion of singular Hermitian metrics.

\begin{definition}[{see \cite[Definition 17.1]{HPS18}}]
Let \(X\) be a complex manifold and let \(E\) be a  vector bundle on \(X\) of rank \(r\geq 1\).
A singular Hermitian metric on \(E\) is a function \(h\) assigning to each point \(x\in X\) a singular Hermitian inner product 
$$|-|_{h,x}\colon E_x\to[0,+\infty]$$ on the complex vector space \(E_x\), satisfying the following two conditions:
\begin{enumerate}
\item \(h\) is finite and positive definite almost everywhere; that is, for every \(x\) outside a set of measure zero, \(|-|_{h,x}\) is a Hermitian inner product on \(E_x\).
\item \(h\) is measurable; that is, the function
\[
|s|_h\colon U\to[0,+\infty],~~~x\mapsto|s(x)|_{h,x},
\]
is measurable for every open subset \(U\subseteq X\) and every \(s\in H^0(U,E)\).
\end{enumerate}
\end{definition}

\begin{definition}[{see \cite[Definition 19.1]{HPS18}}]\label{d:metric}
Let \(X\) be a normal compact K\"ahler variety.
Let \(\mathcal{E}\) be a torsion-free sheaf on \(X\), and set 
\(X_0\coloneqq X_{\textup{reg}}\cap X_{\mathcal{E}}\), 
where \(X_{\textup{reg}}\) is the nonsingular locus of \(X\) and 
\(X_{\mathcal{E}}\) is the maximal open subset over which \(\mathcal{E}\) is locally free. 
Since \(X\) is normal and \(\mathcal{E}\) is torsion-free, 
the subset \(X_0\subseteq X\) is a Zariski open subset 
whose complement has codimension at least \(2\) in \(X\). 
A singular Hermitian metric \(h\) on \(\mathcal{E}\) is defined to be a singular Hermitian metric on the vector bundle \(\mathcal{E}|_{X_0}\).
\end{definition}

Let \(\theta\) be a smooth \((1,1)\)-form on \(X\) with local potentials; that is, 
\(\theta\) can be written locally as \(\theta=dd^c\varphi\) in a neighborhood of every point of \(X\).
We write
\[
\sqrt{-1}\Theta_h\geq \theta\otimes\id_\mathcal{E}
\]
if, for every open subset \(U\subseteq X\) and every local section 
\(e\in H^0(U\cap X_0,(\mathcal{E}|_{X_0})^\vee)\), 
the function \(\log|e|_{h^\vee}-\varphi\) is psh on \(U \cap X_0\), 
where \(\varphi\) is a local potential of \(\theta\) and \(h^\vee\) is the induced metric on the dual bundle \((\mathcal{E}|_{X_0})^\vee\).
Note that the function \(\log|e|_{h^\vee}-\varphi\) extends to a psh function on $U$ 
since $ \codim (X \setminus X_0) \geq 2$.

\subsection{Singular positivity of torsion-free sheaves}\label{sec-psef}
In this subsection, we recall positivity notions for torsion-free sheaves,
using the notation introduced in Definition \ref{d:metric}.

\begin{definition}\label{df-pos}
Fix a K\"ahler form \(\omega_X\) on a normal compact K\"ahler variety \(X\), 
that is, a positive $(1,1)$-form with local potentials.

\begin{enumerate}
\item A torsion-free sheaf \(\mathcal{E}\) on \(X\) is \textit{pseudo-effective} if, for every \(m\in\mathbb{Z}_+\), there exists a singular Hermitian metric \(h_m\) on the \(m\)-th symmetric power \(\textup{Sym}^m\mathcal{E}|_{X_0}\) such that \(\sqrt{-1}\Theta_{h_m}\geq-\omega_X\otimes\id_{\textup{Sym}^m\mathcal{E}}\). 

\item 
A torsion-free sheaf \(\mathcal{E}\) on \(X\) is \textit{weakly positively curved} 
if, for every $\varepsilon >0$, there exists a singular Hermitian metric $h_\varepsilon$ on $\mathcal{E}$
such that \[
\sqrt{-1}\Theta_{h_\varepsilon}\geq -\varepsilon \omega_X \otimes\id_\mathcal{E}. 
\]

\item 
A torsion-free sheaf \(\mathcal{E}\) on \(X\) is \textit{positively curved} 
if there exists a singular Hermitian metric $h_\varepsilon$ on $\mathcal{E}$
such that \[
\sqrt{-1}\Theta_{h_\varepsilon}\geq 0\otimes\id_\mathcal{E}. 
\]
\end{enumerate}
\end{definition}

By definition, every weakly positively curved sheaf is pseudo-effective. 
Indeed, for each \(m\in\mathbb{Z}_+\), choose \(\varepsilon<1/m\) and take a singular Hermitian metric \(h_\varepsilon\) on \(\mathcal{E}\) as in (2). 
Then the induced singular Hermitian metric 
\(h_m\coloneqq\textup{Sym}^m h_\varepsilon\) 
on \(\textup{Sym}^m\mathcal{E}|_{X_0}\) satisfies
\[
\sqrt{-1}\Theta_{h_m}
=
\sqrt{-1}\Theta_{\textup{Sym}^m h_\varepsilon}
\geq -m\varepsilon\omega_X\otimes\id_{\textup{Sym}^m\mathcal{E}}
\geq 
-\omega_X\otimes\id_{\textup{Sym}^m\mathcal{E}}.
\]
For projective varieties, using ample divisors, 
the pseudo-effectivity of \(\mathcal{E}\) can be characterized as follows. 

\begin{proposition}[{see \cite[Proposition 2.3]{Mat23}}]
Let \(\mathcal{E}\) be a torsion-free sheaf on a projective variety \(X\). 
Then the following conditions are equivalent:
\begin{itemize}
\item \(\mathcal{E}\) is pseudo-effective in the sense of Definition \ref{df-pos}. 
\item There exists an ample line bundle \(A\) such that, for every \(m\in\mathbb{Z}_+\), 
the reflexive hull \(\textup{Sym}^{[m]}\mathcal{E}\otimes A\) is generically globally generated. 
\end{itemize}
Moreover, when \(\mathcal{E}\) is locally free and $X$ is smooth, 
the above conditions are equivalent to the following condition:
\begin{itemize}
\item the non-nef locus of the hyperplane line bundle 
\(\mathcal{O}_{\mathbb{P}(\mathcal{E})}(1)\) does not dominate \(X\).
\end{itemize}
\end{proposition}

The pseudo-effectivity of the tautological line bundle 
\(\mathcal{O}_{\mathbb{P}(\mathcal{E})}(1)\) 
is much weaker than the pseudo-effectivity of \(\mathcal{E}\) in the sense of Definition \ref{df-pos}. 
There are many examples where the former holds but the latter fails; 
for instance, the vector bundle \(\mathcal{O}\oplus\mathcal{O}(-1)\) over \(\mathbb{P}^1\) has this property. 
For an example involving tangent bundles, we refer to Example \ref{ex-psef}, 
where \(\mathcal{O}_{\mathbb{P}(T_X)}(1)\) is pseudo-effective 
but \(T_X\) is not pseudo-effective in the sense of Definition \ref{df-pos}.

We finally recall the notion of generically ample sheaves.

\begin{definition}[{see \cite[Definition 2.1]{Mat23}}]\label{d:psef}
A torsion-free sheaf \(\mathcal{E}\) on a normal projective variety \(X\) is \textit{generically ample} if, for any ample Cartier divisors \(A_1,\dots,A_{n-1}\), the minimal slope satisfies
$
\mu_{A_1,\dots,A_{n-1}}^{\min}(\mathcal{E})>0.
$
Here, the minimal slope is defined by
\[
\mu_{A_1,\dots,A_{n-1}}^{\min}(\mathcal{E})
\coloneqq
\inf_{}
\left\{
\frac{c_1(\mathcal{Q})\cdot A_1\cdots A_{n-1}}{\textup{rk}\mathcal{Q}}
\ \middle|\ 
\mathcal{Q} \text{ is a torsion-free quotient of } \mathcal{E}
\right\}.
\]
\end{definition}

\subsection{Numerical dimension and divisorial Zariski decompositions}\label{s:nd-prelim}

\begin{definition}
Let \(X\) be a compact K\"ahler manifold and let \(\alpha\) be a pseudo-effective real \((1,1)\)-class. 
\begin{enumerate}
\item The numerical dimension is defined by
\[
\textup{nd}_X(\alpha)
\coloneqq
\max\{0\leq p\leq\dim X\mid\langle\alpha^p\rangle\neq0\},
\]
where \(\langle\alpha^p\rangle\) is Demailly's mobile product; see \cite[Definition 18.13]{Dem12}. 
\item Boucksom's divisorial Zariski decomposition is given by
\[
\alpha=\{N_X(\alpha)\}+P_X(\alpha),
~P_X(\alpha)=\langle\alpha\rangle,
\]
where \(N_X(\alpha)\) is an effective real divisor and \(P_X(\alpha)\) is modified nef. 
\end{enumerate}
\end{definition}
By \cite[Theorems 18.12(d) and 18.14]{Dem12} and \cite{Bou04}, we know that
\[
\textup{nd}_X(\alpha)=0
~\Leftrightarrow~
P_X(\alpha)=0
~\Leftrightarrow~
\alpha=\{N_X(\alpha)\}.
\]
\begin{definition}
Let \(X\) be a normal compact K\"ahler space, and let \(H^{1,1}_{\mathrm{BC}}(X)\) be the Bott--Chern cohomology defined in \cite[Definition 4.6.2]{BG13}. 
If \(\alpha\in H^{1,1}_{\mathrm{BC}}(X,\mathbb{R})\) is pseudo-effective and \(\mu\colon Y\to X\) is a resolution, set
\[
N_X(\alpha)\coloneqq\mu_*N_Y(\mu^*\alpha),
\]
which is independent of the resolution by \cite[Definition A.6 and Lemma A.7]{DHY23}. 
We call \(\alpha\) modified nef if \(N_X(\alpha)=0\).
\end{definition}

\begin{proposition-definition}\label{prop-defn-num}
Let \(X\) be a normal compact K\"ahler space with rational singularities, let \(\alpha\in H^{1,1}_{\mathrm{BC}}(X,\mathbb{R})\) be pseudo-effective, and let \(\mu\colon Y\to X\) be a resolution.
\begin{enumerate}
\item The number
\[
\textup{nd}_X(\alpha)\coloneqq\textup{nd}_Y(\mu^*\alpha)
\]
is independent of \(\mu\), which we call the numerical dimension of \(\alpha\).
\item If \(E\geq0\) is a \(\mu\)-exceptional real divisor, then
\[
\textup{nd}_Y(\mu^*\alpha+[E])=\textup{nd}_X(\alpha).
\]
\item If \(\alpha\) is modified nef and \(\textup{nd}_X(\alpha)=0\), then \(\alpha=0\) in \(H^{1,1}_{\mathrm{BC}}(X,\mathbb{R})\).
\item If \(X\) is smooth and \(\alpha\) is nef, then
\[
\textup{nd}_X(\alpha)=\max\{p\mid\alpha^p\neq0\}.
\]
On a normal compact K\"ahler space with rational singularities, a nef
class of numerical dimension zero is zero, by \textup{(3)}.
\end{enumerate}
\end{proposition-definition}

\begin{proof}
For (1), let \(p\colon W\to Y\) be a modification between compact K\"ahler manifolds and let \(\beta\) be pseudo-effective on \(Y\). We first prove, for every \(k\geq0\),
\begin{equation}\label{eq:push}
p_*\langle(p^*\beta)^k\rangle=\langle\beta^k\rangle,
~
\langle(p^*\beta)^k\rangle=0\Longleftrightarrow\langle\beta^k\rangle=0.
\end{equation}
When \(\beta\) is big, choose a current \(T_{\min}\in\beta\) with minimal singularities. 
The transformation rule for non-pluripolar products gives the first equality in \eqref{eq:push}; see \cite[Proposition 1.12, Remark 1.7, and Definition 1.17]{BEGZ10}.

For a general pseudo-effective \(\beta\), fix K\"ahler classes \(\omega\) on \(Y\) and \(\widetilde\omega\) on \(W\). 
Since \(p^*\omega\) is nef and big, there are constants \(a,b>0\) such that
\(p^*\omega-a\widetilde\omega\) and \(b\widetilde\omega-p^*\omega\) 
are pseudo-effective; see \cite[Theorem 0.5]{DP04}. Monotonicity of mobile products gives
\[
\left\langle(p^*\beta+a\delta\widetilde\omega)^k\right\rangle
\preceq
\left\langle(p^*\beta+\delta p^*\omega)^k\right\rangle
\preceq
\left\langle(p^*\beta+b\delta\widetilde\omega)^k\right\rangle.
\]
Letting \(\delta\to 0\), applying the big case to \(\beta+\delta\omega\), and passing to the limit proves the first equality in the display \eqref{eq:push}.

For the nonvanishing part, put
\[
A\coloneqq\langle(p^*\beta)^k\rangle,~
B\coloneqq\langle\beta^k\rangle=p_*A,~
m\coloneqq\dim W-k.
\]
Suppose that \(A\neq0\). If \(m=0\), then \(B\neq0\) follows immediately. If \(m>0\), choose \(c>0\) such that \(cp^*\omega-\widetilde\omega\) is pseudo-effective. Mobile products are movable, intersections with nef classes remain movable, and a pseudo-effective \((1,1)\)-class pairs non-negatively with a movable class; see \cite[Propositions 2.24 and 2.26]{Xia26}. Hence
\[
0<A\cdot\widetilde\omega^m
\leq c^mA\cdot(p^*\omega)^m
=c^mB\cdot\omega^m,
\]
so \(B\neq0\). The converse follows from \(B=p_*A\).

Now let \(\mu_j\colon Y_j\to X\), \(j=1,2\), be two resolutions and let \(p_j\colon W\to Y_j\) be a common resolution. Applying the display \eqref{eq:push} to \(p_1\) and \(p_2\) proves (1).

Part (2) follows from (1), \cite[Lemma A.5]{DHY23}, and \cite[Proposition 3.2.10]{Bou02}. For (3), if \(\alpha\) is modified nef and has numerical dimension zero, then \(\mu^*\alpha=\{N_Y(\mu^*\alpha)\}\) by \cite[Theorem 18.14]{Dem12}; consequently,
\[
\alpha=\mu_*\mu^*\alpha
=\mu_*\{N_Y(\mu^*\alpha)\}
=\{N_X(\alpha)\}=0.
\]
Finally, when \(X\) is smooth and \(\alpha\) is nef, its mobile products agree with its ordinary products by \cite[Theorem 18.12(b)]{Dem12}. On a normal compact K\"ahler space with rational singularities, a nef class is modified nef, so the final assertion follows from (3).
\end{proof}

\begin{lemma}\label{l:nd}
Let \(\phi\colon(X,\Delta)\dashrightarrow(X',\Delta')\) be one step of a
\((K_X+\Delta)\)-MMP, either a divisorial contraction or a flip, between
strongly \(\mathbb{Q}\)-factorial compact K\"ahler klt pairs, where
\(\Delta'=\phi_*\Delta\). Assume that \(K_X+\Delta\) is pseudo-effective.
Then
\[
\textup{nd}_X(K_X+\Delta)
=\textup{nd}_{X'}(K_{X'}+\Delta').
\]
\end{lemma}

\begin{proof}
Choose a common resolution
\[
\xymatrix{&W\ar[dl]_p\ar[dr]^q&\\
X\ar@{-->}[rr]^\phi&&X'.}
\]
By the negativity lemma for the \((K_X+\Delta)\)-negative MMP, we have
\[
p^*(K_X+\Delta)=q^*(K_{X'}+\Delta')+F,
\]
where \(F\) is an effective \(q\)-exceptional divisor. Since
pseudo-effectivity is preserved by birational pushforward,
\(K_{X'}+\Delta'\) is pseudo-effective. Proposition-Definition
\ref{prop-defn-num} (1) and (2) now give the asserted equality.
\end{proof}

\subsection{Int-amplified endomorphisms}
The following characterization is due to Meng, together with an observation of Matsuzawa in the projective case; see \cite{Men20,MZ23-survey}. 
If one of the equivalent conditions holds, we say that the endomorphism \(f\) is \textit{int-amplified}.

\begin{proposition}[{see \cite[Theorem 1.1]{Zho21-Asian} and \cite[Proposition 2.10]{DZ23}}]\label{p:int}
Let \(f\colon X\to X\) be a surjective endomorphism of a compact klt K\"ahler variety. 
Then the following conditions are equivalent.
\begin{enumerate}
\item There exists a K\"ahler class \([\omega]\) such that \(f^*[\omega]-[\omega]\) is also a K\"ahler class.
\item All the eigenvalues of the linear operator $\varphi\coloneqq f^*|_{H^{1,1}(X,\mathbb{R})}$ have modulus greater than $1$.
\item There exists a big $(1,1)$-class $[\theta]$ such that $f^*[\theta]-[\theta]$ is big; that is, it can be represented by a K\"ahler current $T$.
\item For every nonempty nonzero $\varphi$-invariant convex cone $C$ in $H^{1,1}(X,\mathbb{R})$, one has
$\emptyset\neq (\varphi-\text{id})^{-1}(C)\subseteq C$.
\end{enumerate}
Moreover, if \(X\) is smooth, then the above conditions are also equivalent to the following condition. 
\begin{enumerate}
\item[(5)] The endomorphism  \(f\) has a dominant topological degree; that is, its topological degree is strictly larger than all the other dynamical degrees.
\end{enumerate}
\end{proposition}
We note that, since \(X\) has only klt and hence rational singularities, the endomorphism in Proposition \ref{p:int} is finite;  see \cite[Lemma 2.10]{Zho21-Asian}. 

There are two elementary examples of int-amplified endomorphisms.
First, any non-isomorphic surjective endomorphism of a normal projective variety of Picard number one is int-amplified.
Second, the product of two int-amplified endomorphisms is again int-amplified.
See \cite{Men20,MZ23-survey,Zho21-Asian} and the references therein for more details.

\subsection{Quasi-\'etale covers}\label{sec-qe}
Let \(X\) be a normal compact K\"ahler variety.
A finite cover \(\widehat{X}\to X\) is called \textit{quasi-\'etale} if it is \'etale in codimension one. 
The \textit{augmented irregularity} \(\widehat{q}(X)\) 
of a compact klt K\"ahler variety \(X\) is defined by
\[
\widehat{q}(X)
\coloneqq
\sup \left\{
q(\widehat{X})
=
h^1(\widehat{X},\mathcal{O}_{\widehat{X}})
\ \middle|\ 
\widehat{X}\to X \text{ is a finite quasi-\'etale cover}
\right\}.
\]
A normal compact K\"ahler variety \(X\) is called \textit{maximally quasi-\'etale} if the homomorphism of algebraic fundamental groups
\[
\widehat{\pi}_1(X_{\textup{reg}})\to\widehat{\pi}_1(X)
\]
induced by \(i_*\colon\pi_1(X_{\textup{reg}})\to\pi_1(X)\) is an isomorphism, where \(i\colon X_{\textup{reg}}\to X\) is the natural inclusion. 
By \cite[Theorem 2.5]{FO25} (cf.~\cite[Theorem 1.5]{GKP16-duke}), every compact klt K\"ahler variety admits a finite quasi-\'etale cover \(\widehat{X}\to X\) such that \(\widehat{X}\) is maximally quasi-\'etale. 
In particular, every finite quasi-\'etale cover of \(\widehat{X}\) is \'etale. 
In this case, we also call \(\widehat{X}\to X\) a \textit{maximally quasi-\'etale cover} of \(X\).
Note that maximally quasi-\'etale covers are not necessarily unique.

We end this section with the following theorem of Cao--Deng--Matsumura \cite[Theorem 1.1]{CDM26} 
(cf.~\cite{HP19,IMZ23}), which will be used in the proof of Theorem \ref{t:psef-rc}.
\begin{theorem}\label{t:CDM}
Let \(X\) be a maximally quasi-\'etale compact klt K\"ahler variety.
Assume that \(\mathcal{E}\) is a pseudo-effective reflexive sheaf with vanishing first Chern class \(c_1(\mathcal{E})=0\).
Then \(\mathcal{E}\) is locally free and numerically flat on $X$.
\end{theorem}

\section{Dynamical positivity of tangent sheaves: Proof of Theorem \ref{t:pos}}
This section is devoted to the proof of Theorem \ref{t:pos-refine} below, which is a refinement of 
Theorem \ref{t:pos}. 
We also give an application on the dynamical characterization of \(Q\)-complex tori (see Proposition \ref{p:Q-torus}).

\begin{theorem}\label{t:pos-refine}
Let $X$ be a normal compact K\"ahler variety admitting an int-amplified endomorphism 
$f\colon X \to X$, and let $\omega$ be a K\"ahler form on $X$. 
Then, for every $\varepsilon>0$, 
there exists a singular Hermitian metric $h_\varepsilon$ on the tangent sheaf $T_X$
satisfying the following properties: 
\begin{itemize}
\item $\sqrt{-1}\Theta_{h_\varepsilon}\geq -\varepsilon \omega \otimes\textup{Id}_{T_X}$; 

\item $h_\varepsilon$ is a smooth Hermitian metric on a Zariski open subset $X_\varepsilon$ 
whose complement $X \setminus X_\varepsilon$ is contained in 
\[
X_{\textup{sing}}\cup \textup{Sing}(f^{s_\varepsilon}) 
\]
for some integer $s_\varepsilon$, 
where $\textup{Sing}(f^{s_\varepsilon})$ denotes the non-smooth locus of $f^{s_\varepsilon}\colon X\to X$.
\end{itemize}

In particular, the tangent sheaf $T_X$ is pseudo-effective. 
Moreover, if $X$ is smooth, 
then the image of the non-nef locus of 
\(\mathcal{O}_{\mathbb{P}(T_X)}(1)\)
under the natural projection $\mathbb{P}(T_X) \to X$
is contained in the union $\cup_{s \in \mathbb{N}} \textup{Sing}(f^s)$
of the non-smooth loci of $f^s\colon X\to X$.
\end{theorem}
\begin{proof}[Proof of Theorem \ref{t:pos-refine}]\label{s:proof}
Let \(X\) be a normal compact K\"ahler variety, and let \(f\colon X\to X\) be an int-amplified endomorphism.
Fix a K\"ahler form \(\omega\) on $X$ such that 
the class \(f^*[\omega]-[\omega]\) contains a K\"ahler form.
By the openness of the K\"ahler cone, we can choose a real number \(r>1\) such that 
\([f^*\omega-r\omega]\) also contains a K\"ahler form, say \(\omega'\).  
Then, for every \(s\in\mathbb{N}\), there exists a globally defined smooth real-valued function \(G_s\) such that 
\begin{align}\label{eq-1}
(f^s)^*\omega-r^s\omega
=
\left(\sum_{i=0}^{s-1}r^{s-1-i}(f^i)^*\omega'\right)+dd^cG_s.
\end{align}
Let \(X_0\) be the big open subset as in Definition \ref{d:metric}.
By \cite[Lemma 3.3]{Ou22}, 
there exists a smooth Hermitian metric \(g\) on \(T_X|_{X_0}\) 
such that 
\begin{align}\label{eq-2}
\sqrt{-1}\Theta_{g}(T_X)
\geq 
-C\omega\otimes\textup{Id}_{T_X}
\end{align}
holds on \(X_0\) for some positive constant \(C\). 

For each \(s\in\mathbb{N}\), there is a natural morphism
\[
(f^s)^* \Omega_X\to \Omega_X
\]
defined on $X$. 
Taking double duals, we obtain an injective sheaf morphism 
\[
0\to T_X \to (f^s)^{[*]}T_X
:=
\left((f^s)^*T_X|_{X_0}\right)^{\vee\vee}. 
\]
By construction, this morphism is an isomorphism over 
\(X_0 \setminus \textup{Sing}(f^s)\), 
that is, over the smooth locus of \(f^s\).
In particular, it is generically an isomorphism.
Let \(g_s\) be the singular Hermitian metric on \((f^s)^{[*]}T_X\) induced by the above morphism and by the metric \(g\). 
By construction, the Hermitian metric \(g_s\) is smooth on 
\[
X_0 \setminus \textup{Sing}(f^s).
\]
Moreover, we obtain 
\[
\sqrt{-1}\Theta_{g_{s}}((f^s)^{[*]}T_X)
\geq 
-C\omega\otimes\textup{Id}_{(f^s)^{[*]}T_X} 
\]
by \eqref{eq-2} and the standard curvature property of quotient metrics; 
see \cite[Proof of Lemma 2.4.3 (3)]{PT18}.

Fix \(\varepsilon>0\) and choose \(s=s_\varepsilon\) sufficiently large so that
$
\varepsilon r^s>C.
$
We twist the metric \(g_s\) by the weight \(\exp(-\varepsilon G_s)\). 
Using \eqref{eq-1}, we obtain
\begin{align*}
&\sqrt{-1}\Theta_{g_{s}\exp(-\varepsilon G_s)}
\bigl((f^s)^{[*]}T_X\bigr)
+\varepsilon (f^s)^*\omega\otimes\textup{Id}_{(f^s)^{[*]}T_X} \\
&\geq
\left((\varepsilon r^s-C)\omega
+\sum_{i=0}^{s-1}r^{s-1-i}\varepsilon (f^i)^*\omega'\right)
\otimes\textup{Id}_{(f^s)^{[*]}T_X}
\geq 0.
\end{align*}
This means that, if \(U\subseteq X\) is an open subset and
$
w\in H^0\bigl(U,((f^s)^{[*]}T_X)^\vee\bigr),
$
then the function
\[
\log|w|_{(g_s\exp(-\varepsilon G_s))^\vee}
+
\varepsilon (f^s)^*\phi
\]
is plurisubharmonic on \(U\), where \(\phi\) is a local potential of \(\omega\). 
Here
\[
(g_s\exp(-\varepsilon G_s))^\vee
=
g_s^\vee\exp(\varepsilon G_s)
\]
denotes the induced metric on the dual sheaf.

We now construct a singular Hermitian metric on \(T_X\) by averaging over the finite fibres of \(f^s\). 
Let \(e\in H^0(U,T_X^\vee)\) be a local section. 
Then \((f^s)^*e\) is a local section of \(((f^s)^{[*]}T_X)^\vee\) over \(f^{-s}(U)\). 
The function
\[
|(f^s)^*e|_{(g_s\exp(-\varepsilon G_s))^\vee}
\]
is a priori defined on \(f^{-s}(U)\cap X_0\), 
but extends to $f^{-s}(U)$ 
by the plurisubharmonicity obtained above. 
We define a function \(\varphi_e\) on \(U\) by
\[
\varphi_e(p)
\coloneqq
\sum_{q\in f^{-s}(p)}
|(f^s)^*e|^2_{(g_s\exp(-\varepsilon G_s))^\vee}(q),
\]
where the sum is taken with multiplicities.
Since \(f^s\) is finite, this is a finite sum. 
Moreover, on the open set
$
X_0\setminus \textup{Sing}(f^s),
$
the map \(f^s\) is smooth and the metric \(g_s\) is smooth; hence \(\varphi_e\) is smooth there.

We claim that the functions \(\varphi_e\) define a singular Hermitian metric \(\chi\) on \(T_X\).
More precisely, we define the dual metric \(\chi^\vee\) on \(T_X^\vee\) by
\[
|e|^2_{\chi^\vee,x}
\coloneqq
\varphi_e(x)
\]
and we show that this metric $\chi^\vee $ is a singular Hermitian metric with the desired properties. 
Fiberwise, $\chi^\vee $ gives a Hermitian squared norm,  
because \(\varphi_e(x)\) is a finite sum of Hermitian squared norms. 
Thus, the first condition in Definition \ref{d:metric} is satisfied. 
The measurability condition also follows from the construction, since a finite sum of measurable functions is measurable. 
Therefore \(\chi\) is a singular Hermitian metric on \(T_X\).

It remains to check the positivity lower bound. 
For every local section \(e\in H^0(U,T_X^\vee)\), we know that
\[
\log|(f^s)^*e|_{(g_s\exp(-\varepsilon G_s))^\vee}
+
\varepsilon (f^s)^*\phi
\]
is plurisubharmonic on \(f^{-s}(U)\). 
Hence
\[
\exp\left(
2\varepsilon (f^s)^*\phi
\right)
|(f^s)^*e|^2_{(g_s\exp(-\varepsilon G_s))^\vee}
\]
is locally the exponential of a plurisubharmonic function. 
By the standard trace argument for finite morphisms, together with \cite[Proposition 1.13]{Dem85}, we obtain that
\[
\log\left(
(f^s)_*
\left(
\exp\left(2\varepsilon (f^s)^*\phi\right)
|(f^s)^*e|^2_{(g_s\exp(-\varepsilon G_s))^\vee}
\right)
\right)
\]
is plurisubharmonic on \(U\).
By the definition of \(\varphi_e\), we have
\[
(f^s)_*
\left(
\exp\left(2\varepsilon (f^s)^*\phi\right)
|(f^s)^*e|^2_{(g_s\exp(-\varepsilon G_s))^\vee}
\right)
=
\exp(2\varepsilon\phi)\varphi_e.
\]
Therefore, the function
\[
\log|e|_{\chi^\vee}+\varepsilon\phi
=
\frac{1}{2}\log\varphi_e+\varepsilon\phi
\]
is psh. 
This is precisely the curvature inequality
\[
\sqrt{-1}\Theta_{\chi}
\geq
-\varepsilon\omega\otimes\textup{Id}_{T_X}.
\]

Finally, by construction, the metric \(\chi\) is smooth on the Zariski open subset
$
X_\varepsilon
\coloneqq
X_0\setminus \textup{Sing}(f^{s_\varepsilon}).
$
Thus \(\chi\) satisfies both conditions of Theorem \ref{t:pos-refine}. 
Setting \(h_\varepsilon\coloneqq \chi\), we obtain the desired singular Hermitian metric on \(T_X\).
\end{proof}

We end this section with the following application of Theorems \ref{t:pos} and \ref{t:CDM}.
\begin{proposition}[{cf.~\cite[Theorem 1.21]{GKP16-duke},\cite[Lemma 6.1]{Men20}}]\label{p:Q-torus}
Let \(f\colon X\to X\) be an int-amplified endomorphism of an \(n\)-dimensional normal \(\mathbb{Q}\)-Gorenstein  compact K\"ahler variety.
Assume that either \(X\) is klt, or \(X\) is non-uniruled.
If the canonical divisor \(K_X\) is pseudo-effective, then \(X\) is a quasi-\'etale quotient of a complex torus.
\end{proposition}

\begin{proof}
By Theorem \ref{t:pos}, the tangent sheaf \(T_X\) is pseudo-effective and hence its determinant \(-K_X\) is pseudo-effective as well.
Therefore, \(K_X\equiv 0\) holds.
Suppose first that \(X\) is klt and let \(\gamma\colon\widehat{X}\to X\) be a maximally quasi-\'etale cover.
The reflexive pullback \(T_{\widehat{X}}\cong (\gamma^*T_X)^{\vee\vee}\) remains pseudo-effective and has vanishing first Chern class.
By Theorem  \ref{t:CDM}, it is locally free and numerically flat so that our result follows. 
In the following, we assume that \(X\) is non-uniruled and we aim to show that \(X\) has only canonical singularities so that we can apply Theorem \ref{t:CDM} again.
Let \(\pi\colon\widetilde{X}\to X\) be a resolution of singularities so that \(\widetilde{X}\) is a non-uniruled compact K\"ahler manifold.
By \cite{Ou25}, the canonical divisor \(K_{\widetilde{X}}\) is pseudo-effective, and hence it admits a divisorial Zariski decomposition (see \cite{Bou04}):
\[
K_{\widetilde{X}}\equiv P+N,
\]
where \(P=P(K_{\widetilde{X}})\) is modified nef in the sense that \(P|_M\) is pseudo-effective for any prime divisor \(M\) and \(N=N(K_{\widetilde{X}})\geq 0\) is the negative part.
Fix any K\"ahler class \(\alpha\) on \(X\).
Since \(\pi_*K_{\widetilde{X}}=K_X\equiv 0\), we obtain from the projection formula that
\[
0=K_{\widetilde{X}}\cdot\pi^*\alpha^{n-1}=P\cdot\pi^*\alpha^{n-1}+N\cdot\pi^*\alpha^{n-1}.
\]
As both \(P\) and \(N\) are pseudo-effective, we get the vanishing \[
P\cdot\pi^*\alpha^{n-1}=N\cdot\pi^*\alpha^{n-1}=0.
\]
Since \(N\) is effective and \(\alpha\) is K\"ahler, each component of \(N\) is \(\pi\)-exceptional. 

We claim that \(P\equiv 0\).
By the analytic version of Kodaira's lemma (see, e.g., \cite[Lemma 2.7]{Zho23}), with \(\widetilde{X}\) further replaced a higher resolution, we can write \(\pi^*\alpha=\omega+[F]\) where \(\omega\) is some K\"ahler class and the support of \(F\geq0\) coincides with the non-K\"ahler locus of \(\pi^*\alpha\) which is \(\pi\)-exceptional. 
As \(P\) is modified nef, the vanishing \(P\cdot\pi^*\alpha^{n-1}=0\) implies \(P\cdot \omega\cdot \pi^*\alpha^{n-2}=0\). 
Inductively, we have \(P\cdot\omega^{n-1}=0\).
By the Hodge--Riemann theorem, \(P^2\cdot\omega^{n-2}\leq 0\) with the equality holding if and only if \(P\equiv 0\).
Fix a sufficiently large \(b\) such that \([P]+b\omega\) is a K\"ahler class.
Then 
\[
P^2\cdot\omega^{n-2}=P(P+b\omega)\omega^{n-2}\geq 0;\]
consequently, \(P\equiv 0\).
Now, we have 
\[
N\equiv K_{\widetilde{X}}=\pi^*K_X+E_1-E_2\equiv E_1-E_2.
\]
Therefore, \(E_1\equiv E_2+N\).
By possibly further canceling some common components of \(N\) and \(E_1\), we have \(E_1'\equiv E_2'\) such that \(E_1'\) and \(E_2'\) are positive linear combinations of \(\pi\)-exceptional divisors with no common components. Moreover, \(E_2'\geq E_2\) in the sense of currents. 
Then \(E_1'\cdot E\) is a positive closed \((2,2)\)-current for any \(\pi\)-exceptional divisor \(E\); hence \(E_1'\) is modified nef. 
By the same argument, we obtain \(E_2'\equiv 0\) 
and hence \(0\leq E_2\leq E_2'\equiv 0\). 
In particular, \(E_2=0\), and  \(X\) has canonical singularities as desired.
\end{proof}

\section{Rational connectedness and pseudo-effective tangent sheaves}\label{s:RC}
After preparing a few preliminary results, we prove Theorem \ref{t:psef-rc} in this section. 
We thank Masataka Iwai for communicating the ideas behind Propositions \ref{prop-virtually} and \ref{prop-irreg}.

We begin with the following result, 
which extends \cite[Theorem 2.1]{MQ25} and is based on the spirit of \cite{Mat25}.

\begin{proposition}\label{prop-virtually}
Let \(X\) be a compact klt K\"ahler variety with pseudo-effective tangent sheaf \(T_X\).
Then the image of any linear representation of its fundamental group \(\pi_1(X)\to\textup{GL}(n,\mathbb{C})\) is virtually abelian. 
\end{proposition}
\begin{proof}
First, we recall that a compact K\"ahler variety is called \textit{special} in the sense of Campana if some resolution 
\(\sigma\colon\widetilde{X}\to X\) satisfies the following condition: the Iitaka dimension satisfies 
\(\kappa(\widetilde{X},L)<p\) for every \(p>0\) and every line bundle
\(L\subseteq \wedge^p\Omega_{\widetilde{X}}\) 
(see \cite[Definitions 2.1, 2.24 and Theorem 2.27]{Cam04}). 
The proposition then follows from \cite[Theorem 7.8]{Cam04} and \cite[Theorem 1.1]{Tak03} once we show that 
\(\widetilde{X}\) is special under the assumption that \(T_X\) is pseudo-effective.

Suppose, to the contrary, that \(\widetilde{X}\) is not special.
Then there exist an integer \(p>0\) and a line bundle \(L\) 
such that 
\(L\subseteq\wedge^p\Omega_{\widetilde{X}}\) and 
\(\kappa(\widetilde{X},L)\geq p>0\).
Then we have the natural inclusion
\[
\sigma_*L\subseteq \sigma_{[*]}L :=(\sigma_*L)^{\vee\vee}\subseteq\Omega_X^{[p]},
\]
where \(\Omega_X^{[p]}\) denotes the reflexive hull of \(\Omega_X^{p}\). 
Since \(\kappa(\widetilde{X},L)>0\), the line bundle \(L\), and hence the reflexive rank-one sheaf $\sigma_{[*]}L$, 
is pseudo-effective.
On the other hand, by dualizing the inclusion \(\sigma_{[*]}L\subseteq\Omega_X^{[p]}\), we obtain a generically surjective morphism
\[
T_X^{[p]}\to (\sigma_{[*]}L)^\vee.
\]
Since \(T_X^{[p]}\) is pseudo-effective by \cite[Proposition 2.5]{CDM26}, 
it follows that \((\sigma_{[*]}L)^\vee\) is pseudo-effective as well.
Thus \(\sigma_{[*]}L\) is numerically trivial. 
Consequently, \(L\) is numerically equivalent to a \(\sigma\)-exceptional  divisor. 
This contradicts the fact that \(\kappa(\widetilde{X},L)>0\).
Therefore, \(\widetilde{X}\) is special, and the proposition is proved.
\end{proof}

The following proposition extends \cite[Theorem 2.2]{MQ25} from compact K\"ahler manifolds to compact klt K\"ahler varieties, and also extends \cite[Lemma 6.1]{IMZ23} from projective varieties to K\"ahler varieties.
We remark that we cannot simply reduce the proof to a functorial resolution, 
whose tangent bundle need no longer be pseudo-effective (see Example \ref{ex-mqe}).
\begin{proposition}\label{prop-irreg}
Let \(X\) be a compact klt K\"ahler variety with pseudo-effective tangent sheaf \(T_X\). 
Then the augmented irregularity of \(X\) satisfies \(\widehat{q}(X)\leq\dim X\). 
In particular, there exists a finite quasi-\'etale cover \(X'\to X\) such that \(q(X')=\widehat{q}(X)\).
Moreover, if \(\widehat{q}(X)=0\), then \(X\) is projective and rationally connected.
\end{proposition}
\begin{proof}
Since the tangent sheaf \(T_X\) is pseudo-effective, we can apply \cite[Lemma 2.9 (1)]{IMZ23} to conclude that any non-zero section of \(H^0(X,\Omega_X^{[1]})\) is nonvanishing on the smooth locus \(X_{\textup{reg}}\), noting that the proof of loc.~cit. does not use the projectivity assumption. 
Thus, by the proof of \cite[Proposition 3.9]{DPS94}, the restriction of the Albanese morphism 
to the regular locus
\[
\textup{alb}_X|_{X_{\textup{reg}}}\colon X_{\textup{reg}}\to\textup{Alb}(X)
\]
is a submersion.
Hence, the Albanese morphism is surjective.
Applying the same argument to all finite quasi-\'etale covers of \(X\), and using the fact that the pseudo-effectivity of the tangent sheaf is preserved under reflexive finite pullback, we obtain \(\widehat{q}(X)\leq\dim X\). 
Since \(\widehat{q}(X)\leq\dim X<\infty\), we can take a quasi-\'etale cover \(X'\to X\) with \(q(X')=\widehat{q}(X)\).
This proves the first assertion.

We now prove the rational connectedness of \(X\) when \(\widehat{q}(X)=0\). 
After replacing \(X\) with a quasi-\'etale cover, we may assume that \(X\) is maximally quasi-\'etale and satisfies \(q(X)=\widehat{q}(X)=0\).
By Proposition \ref{p:Q-torus}, the canonical divisor \(K_X\) is not pseudo-effective.

By \cite[Theorem 1.1]{Ou25}, \(X\) is covered by rational curves.
Let \(\pi\colon X\dashrightarrow Y\) be a dominant meromorphic MRC fibration.
We shall prove that \(Y\) is a point.
Suppose, to the contrary, that \(\dim Y>0\).
By taking a higher model, we may assume that \(Y\) is smooth.
Let \(\widetilde{X}\) be the resolution of the graph of \(\pi\), giving the following commutative diagram:
\[
\xymatrix{
&\widetilde{X}\ar[dr]^{\widetilde{\pi}}\ar[dl]_{\sigma}&\\
X\ar@{-->}[rr]_\pi&&Y
}
\]
Note also that \(Y\) is not uniruled, and thus the canonical divisor \(K_Y\) is pseudo-effective by \cite[Theorem 1.1]{Ou25}.

We have the natural injection 
\[
0\to \widetilde{\pi}^*\Omega_Y\to\Omega_{\widetilde{X}}.
\]
Pushing it forward by \(\sigma\) and taking the double dual, we obtain an injection
\[
0\to \mathcal{F}\coloneqq (\sigma_*\widetilde{\pi}^*\Omega_Y)^{\vee\vee}\to\Omega_{X}^{[1]}.
\] 
On the one hand, since \(T_X\) is pseudo-effective, it follows that \(\mathcal{F}^\vee\) is pseudo-effective.
On the other hand, \(\mathcal{F}\) coincides with the \(\pi\)-pullback of \(\Omega_Y\) over the locus \(X_0\) where \(\pi\) is a morphism.
Since the complement of \(X_0\) has codimension at least \(2\) in \(X\), it follows that its determinant
\[
c_1(\mathcal{F})=c_1(\sigma_*\widetilde{\pi}^*K_Y)
\]
is also pseudo-effective.
Hence, \(\det\mathcal{F}\) and \(\det\mathcal{F}^\vee\) are both pseudo-effective, and therefore \(c_1(\mathcal{F})=0\).
By Theorem \ref{t:CDM}, the sheaf \(\mathcal{F}\) is locally free and numerically flat; in particular, it is induced by a linear representation
\[
\pi_1(X)\to\textup{GL}(r,\mathbb{C}),
\]
where \(r\) is the rank of \(\mathcal{F}\).

The image of this representation is virtually abelian by Proposition \ref{prop-virtually}. 
Thus, by the universality of the abelianization of the fundamental group, up to a finite \'etale cover, the representation factors through
\[
\pi_1(\widetilde{X})\to H_1(\widetilde{X},\mathbb{Z})\to \textup{GL}(r,\mathbb{C}),
\]
where we use the fact that \(X\) is klt and hence \(\pi_1(X)\cong\pi_1(\widetilde{X})\) (see \cite[Theorem 1.1]{Tak03}).
Moreover, we have 
\[
h^1(\widetilde{X},\mathbb{C})=2h^1(\widetilde{X},\mathcal{O}_{\widetilde X})=2h^1(X,\mathcal{O}_X)=0.
\]
By Hodge theory and the universal coefficient theorem, the cohomology \(H_1(\widetilde{X},\mathbb{Z})\) is a torsion group.
Therefore, after further replacing \(X\) by a finite \'etale cover and taking the base change, we may assume that the representation 
\(\pi_1(X)\to\textup{GL}(r,\mathbb{C})\) is trivial.  
Equivalently, \(\mathcal{F}\cong\mathcal{O}_X^{\oplus r}\).
This yields a contradiction, since we have 
\[
0=h^1(X,\mathcal{O}_X)=h^1(\widetilde{X},\mathcal{O}_{\widetilde{X}})=h^0(\widetilde{X},\Omega_{\widetilde{X}}^1)=h^0(X,\sigma_*\Omega_{\widetilde{X}}^1)=h^0(X,\Omega_X^{[1]})\geq r>0.
\] 
Here, the last equality follows from the reflexivity of \(\sigma_*\Omega_{\widetilde{X}}^1\); see, for example, \cite[Theorem 1.4]{KS21}.
This contradiction shows that \(Y\) is a point.
Thus, \(X\) is rationally connected, and thus projective.
This completes the proof of the proposition.
\end{proof}

With these preparations in place, we now prove Theorem \ref{t:psef-rc}.
\begin{proof}[Proof of Theorem \ref{t:psef-rc}]
\((1)\Rightarrow (2)\). 
Let \(X'\to X\) be a maximally quasi-\'etale Galois cover such that \(X'\) is also klt and rationally connected. 
Note that \(X\) and $X'$  are projective.
By \cite[Proposition 3.7]{IMZ23}, the tangent sheaf \(T_X\) is generically ample if and only if \(T_{X'}\) is generically ample. 
Furthermore, by \cite[Proposition 2.6]{Mat23}, the tangent sheaf \(T_{X'}\) is pseudo-effective. 
It then follows from \cite[Theorem 3.5]{IMZ23} that there exists an exact sequence, the so-called \textit{Fujita decomposition},
\[
0\to\mathcal{E}\to T_{X'}\to\mathcal{Q}\to 0
\]
such that \(\mathcal{E}\) is a pseudo-effective generically ample reflexive sheaf and \(\mathcal{Q}^{\vee\vee}\) is a numerically flat locally free sheaf (in the sense that both \(\mathcal{Q}^\vee\) and \(\mathcal{Q}^{\vee\vee}\) are nef). 
Suppose that \(T_{X'}\) is not generically ample, which is equivalent to saying that \(\mathcal{Q}\) is non-trivial. 
Since \(X'\) is klt and rationally connected, and hence simply connected, we obtain that \(\mathcal{Q}^\vee\) is a direct sum of trivial bundles of positive rank. 
Therefore, the dual injection \(0\to\mathcal{Q}^\vee\to\Omega_{X'}^{[1]}\) implies that \(H^0(X',\Omega_{X'}^{[1]})\neq 0\).
For any resolution \(\sigma\colon\widetilde{X}\to X'\), we have \(\Omega_{X'}^{[1]}=\sigma_*\Omega_{\widetilde{X}}^1\); hence \(H^0(\widetilde{X},\Omega_{\widetilde{X}}^1)\neq 0\), which contradicts the rational connectedness of \(X'\).

\((2)\Rightarrow (3)\). 
Assume that \(X\) is projective and \(T_X\) is generically ample. 
Since \(T_X\) is pseudo-effective, by Proposition \ref{prop-irreg}, there exists 
a maximally quasi-\'etale cover \(X'\to X\) such that the Albanese morphism \(\pi\colon X'\to A\coloneqq\textup{Alb}(X')\) is surjective and \(q(X')=\widehat{q}(X)\).
Applying \cite[Proposition 3.7]{IMZ23} and \cite[Proposition 2.6]{Mat23}, we obtain that \(T_{X'}\) is also pseudo-effective and generically ample. 
Suppose, to the contrary, that \(q(X')>0\).
Then \(\pi\) is non-trivial and \(\dim A>0\). 
Dualizing the injection \(0\to \pi^*\Omega_A\to \Omega_{X'}\), we obtain a natural generically surjective morphism
\(T_{X'}\to \pi^{*}T_A\). 
Since \(T_{X'}\) is generically ample, it follows that
\(\pi^*T_A=\mathcal{O}_{X'}^{\oplus\dim A}\) is generically ample as well, which is absurd.

\((3)\Rightarrow (1)\). 
This follows immediately from Proposition \ref{prop-irreg}, applied also to finite quasi-\'etale covers of \(X\).

Finally, assume further that \(X\) is maximally quasi-\'etale.
Since \(X\) is klt and rationally connected under condition \((1)\), it is simply connected; hence \((1)\Rightarrow (4)\) holds.
Conversely, if \(X\) is simply connected, then it has no non-trivial \'etale cover and hence, by the maximally quasi-\'etale assumption, no non-trivial quasi-\'etale cover. 
By taking a resolution and applying \cite[Theorem 1.1]{Tak03}, we obtain 
\(q(X)=q(\textup{Alb}(X))=0\) (see, e.g., \cite[Theorem 5.2]{Kol93}), which implies (3).
\end{proof}

\section{Equivariant maximally quasi-\'etale covers, Proof of Theorem \ref{thm-lift}}\label{s:rmk}
In this section, we  prove Theorem 1.4 and then slightly extend \cite{Yos21} by taking an equivariant \(\mathbb{Q}\)-factorial model and an equivariant maximally quasi-\'etale cover.  
We also give more concrete descriptions of the fundamental group of the regular locus.
See Subsection \ref{sec-qe} for the notation and terminology used here.
As a consequence, we also prove Corollary \ref{c:simply}.

We begin with Theorem \ref{thm-lift}, which guarantees the existence of lifts of surjective endomorphisms to suitable maximally quasi-\'etale covers. As discussed in Section \ref{s:intro}, this lifting problem is subtle in general, especially without passing to iterates.

\begin{proof}[Proof of Theorem \ref{thm-lift}]
First, note that \(f\) is finite, since \(f\) is surjective (see \cite[Lemma 2.10]{Zho21-Asian}). 
Take a maximally quasi-\'etale cover \(p\colon \widehat X\to X\). 
We shall construct a finite \'etale cover \(\tau\colon Y \to \widehat X\) such that
\(f\) lifts to the composite cover \(q\coloneqq p\circ \tau\colon Y\to X\).

For this purpose, consider the restriction
\[
f \colon U:=X_{\reg}\setminus f^{-1}(X_{\sing}) \to X_{\reg},
\]
noting that  \(f(U)\subset X_{\reg}\). 
Fix a point \(x_{0} \in U\), and set
\(x_{1}:=f(x_{0}) \in X_{\reg}\). 
Since \(f\) is finite, we have
\(\codim(X_{\reg}\setminus U)\geq 2\). 
Thus the homomorphism
\[
j_*\colon \pi_1(U, x_0) \to \pi_1(X_{\reg}, x_0)
\]
induced by the inclusion \(j \colon U \hookrightarrow X_{\reg}\) is an
isomorphism.

Choose a path \(\ell\) in \(X_{\reg}\) connecting \(x_{0}\) and \(x_{1}\),
and let
\[
\alpha_\ell\colon
\pi_1(X_{\reg},x_1)\xrightarrow{\cong}
\pi_1(X_{\reg},x_0)
\]
be the change-of-base-point isomorphism defined by
\(\alpha_\ell([\gamma])=[\ell *\gamma *\ell^{-1}]\). 
Then the map \(f\colon U\to X_{\reg}\) induces a group homomorphism
\[
\varphi\coloneqq
\alpha_\ell\circ f_*\circ j_*^{-1}
\colon
G\coloneqq \pi_1(X_{\reg},x_0) \to G.
\]
Choose a point
\(\widehat x_0\in p^{-1}(x_{0}) \subset \widehat X^\circ:=p^{-1}(X_{\reg})\).

Since \(p\) is quasi-\'etale, the restriction
\(p\colon \widehat X^\circ\to X_{\reg}\) is a finite \'etale cover. 
Let
\[
H\coloneqq
p_*\pi_1(\widehat X^\circ,\widehat x_0)
\subset G
\]
be the  finite-index subgroup corresponding to $p\colon \widehat X^\circ\to X_{\reg}$. 
Consider the core \(H_0\) of
\(H\), defined by
\[
H_0\coloneqq
\bigcap_{\gamma\in G}\gamma H\gamma^{-1}.
\]
Then \(H_0\) is a finite-index normal subgroup of \(G\), and \(H_0\subset H\).
Put \(m\coloneqq [G:H_0]\).
Since \(G=\pi_1(X_{\reg},x_0)\) is finitely generated, it has only finitely
many subgroups of index at most \(m\). 
Moreover, for every \(i\geq 0\), the
subgroup \(\varphi^{-i}(H_0)\) has index at most \(m\). 
Hence the family
$
\{\varphi^{-i}(H_0)\}_{i\geq 0}
$
is finite. 
Therefore the subgroup
\[
K\coloneqq \bigcap_{i\geq 0}\varphi^{-i}(H_0)
\]
is of finite index in \(G\). 
Moreover, we see that \(K\subset H_0\subset H\), and
\(\varphi(K)\subset K\) by construction.

Consider the finite \'etale cover
\[
q^\circ \colon (Y^\circ, y_0)\to (X_{\reg}, x_0)
\]
corresponding to \(K\), namely
\[
(q^\circ)_*\pi_1(Y^\circ,y_0)=K.
\]
Since \(K\subset H\), this cover factors through
\[
q^\circ=p\circ \tau^\circ \colon
Y^\circ \xrightarrow{\tau^\circ} \widehat X^\circ
\xrightarrow{p} X_{\reg}.
\]
Since
$
\codim_{\widehat X_{\reg}}
(\widehat X_{\reg}\setminus \widehat X^\circ)\geq 2,
$
the finite \'etale cover \(\tau^\circ\colon Y^\circ\to \widehat X^\circ\)
corresponds to a finite \'etale cover of \(\widehat X_{\reg}\). 
Since
\(\widehat X\) is maximally quasi-\'etale, this finite \'etale cover extends
to a finite \'etale cover \(\tau\colon Y\to \widehat X\). 
We use the notation
\(q\coloneqq p\circ\tau\colon Y\to X\). 
Then \(Y^\circ=q^{-1}(X_{\reg})\).

We now prove that \(f\) lifts to \(Y\). 
We regard the base point
\(y_0\in Y^\circ\) as a point of \(Y\). 
Let
$
\widetilde \ell
$
be the lift of \(\ell\) to $Y^\circ$ with initial point \(y_0\), and set
\(y_1\coloneqq \widetilde \ell(1)\). 
Then \(q(y_1)=x_1=f(x_0)\).
Let
\[
\beta_{\widetilde \ell}\colon
\pi_1(Y^\circ,y_1)\xrightarrow{\cong}
\pi_1(Y^\circ,y_0)
\]
be the change-of-base-point isomorphism induced by \(\widetilde \ell\).
Then the diagram
\[
\xymatrix{
\pi_1(Y^\circ,y_1) \ar[r]^-{\beta_{\widetilde \ell}} \ar[d]_{q_*}
&
\pi_1(Y^\circ,y_0) \ar[d]^{q_*}
\\
\pi_1(X_{\reg},x_1) \ar[r]^-{\alpha_\ell}
&
\pi_1(X_{\reg},x_0)
}
\]
commutes. 
Since \(q_*\pi_1(Y^\circ,y_0)=K\), it follows that
\[
q_*\pi_1(Y^\circ,y_1)=\alpha_\ell^{-1}(K).
\]
Thus, with respect to the base point \(y_1\), the cover
\(q^\circ\colon Y^\circ\to X_{\reg}\) corresponds to the subgroup
\(\alpha_\ell^{-1}(K)\subset \pi_1(X_{\reg},x_1)\).

Since \(\codim_{Y^\circ}(Y^\circ \setminus q^{-1}(U))\geq 2\), the inclusion
\(q^{-1}(U)\hookrightarrow Y^\circ\) induces an isomorphism
\[
\pi_1(q^{-1}(U),y_0)\xrightarrow{\cong}\pi_1(Y^\circ,y_0).
\]
Under this isomorphism, we have
$
j_*\bigl(q_*\pi_1(q^{-1}(U),y_0)\bigr)=K.
$
Hence, we have 
\[
\alpha_\ell\circ f_*
\bigl(q_*\pi_1(q^{-1}(U),y_0)\bigr)
=
\alpha_\ell\circ f_*\circ j_*^{-1}(K)
=
\varphi(K)
\subset K,
\]
equivalently,
\[
(f\circ q)_*\pi_1(q^{-1}(U),y_0)
\subset
\alpha_\ell^{-1}(K).
\]
Therefore, by the lifting criterion applied to the map
\[
f\circ q\colon (q^{-1}(U), y_0) \to (X_{\reg}, x_1)
\]
and to the finite \'etale cover
\[
q^\circ\colon (Y^\circ, y_1) \to (X_{\reg}, x_1),
\]
we obtain a lift
\[
F^\circ\colon (q^{-1}(U), y_0)\to (Y^\circ, y_1)
\]
such that \(q\circ F^\circ=f\circ q\) on \(q^{-1}(U)\).

It remains to extend \(F^\circ\) to an endomorphism of \(Y\). 
Consider the
fibre product \(P\coloneqq Y\times_X Y\), where the first morphism
\(Y\to X\) is \(f\circ q\), and the second one is \(q\). 
Thus we have the
Cartesian diagram
\[
\xymatrix{
P \ar[r]^-{\pr_2} \ar[d]_{\pr_1} & Y \ar[d]^q \\
Y \ar[r]_-{f\circ q} & X .
}
\]
The morphism \(\pr_1\colon P\to Y\) is finite, since it is obtained by base
change of the finite morphism \(q\colon Y\to X\).

On the open subset \(q^{-1}(U)\subset Y\), the map \(F^\circ\) satisfies
\(q\circ F^\circ=f\circ q\). 
Hence \(F^\circ\) defines a section
\(s^\circ\) of \(\pr_1\) over \(q^{-1}(U)\), namely
\(s^\circ(y)\coloneqq (y,F^\circ(y))\). 
Let \(\Gamma\subset P\) be the
closure of \(s^\circ(q^{-1}(U))\). 
Then the restriction
\(\pr_1|_\Gamma\colon \Gamma\to Y\) is finite and bimeromorphic. 
Indeed, it
is finite because \(\pr_1\) is finite, and it is an isomorphism over the
dense open subset \(q^{-1}(U)\).

Since \(Y\) is normal, it follows that
\(\pr_1|_\Gamma\colon \Gamma\to Y\) is an isomorphism. 
Therefore
\[
F\coloneqq \pr_2|_\Gamma\circ (\pr_1|_\Gamma)^{-1}\colon Y\to Y
\]
is an extension of \(F^\circ\). 
Since \(\Gamma\subset P\), we have
\(q\circ F=f\circ q\). 
Thus \(f\) lifts to the finite cover \(q\colon Y\to X\).

Finally, since \(\tau\colon Y\to \widehat X\) is finite \'etale and
\(\widehat X\to X\) is maximally quasi-\'etale, the composite
\(q\colon Y\to X\) is again a maximally quasi-\'etale cover. 
Replacing
\(\widehat X\) by \(Y\), we obtain the desired cover.
\end{proof}

As an application of Theorem \ref{thm-lift}, together with \cite[Theorem 1.3]{Yos21}, we obtain Theorem \ref{t:qe-mrc} below as a remark on Yoshikawa's cover.

\begin{theorem}\label{t:qe-mrc}
Let \(X\) be a projective klt variety admitting an int-amplified endomorphism \(f\). 
Then, after replacing \(f\) by an iterate, there exists an \(f\)-equivariant maximally quasi-\'etale cover \(\widehat{X}\to X\) such that the following statements hold.
\begin{enumerate}
\item The irregularity attains its maximum: \(q(\widehat{X})=\widehat{q}(X)\).
\item The Albanese morphism \(\pi\colon\widehat{X}\to A\) is a fibration of Fano type with \(\dim A=\widehat{q}(X)\). 
\item A general fibre \(\widehat{F}\) is also maximally quasi-\'etale and the fundamental groups of the regular loci satisfy \(\pi_1(\widehat{X}_\textup{reg})\cong \pi_1(A)\), and \(\pi_1(\widehat{F}_{\textup{reg}})=\{1\}\). 
\end{enumerate} 
\end{theorem}

\begin{proof}
Let \(\phi\colon Y\to X\) be a small \(\mathbb{Q}\)-factorialization of \(X\).
Then it follows from \cite[Theorem 1.2]{MYY24-lcy} (cf.~\cite[Theorem A1.3]{CZ25}) that, after replacing \(f\) by an iterate, \(f\) lifts to a surjective endomorphism \(g\) of \(Y\), which is also int-amplified (see Proposition \ref{p:int}). 
By \cite[Theorem 1.3]{Yos21}, we have the following \(f\)-equivariant commutative diagram:
\[
\xymatrix{
&A\save[]+<1.4pc,0.15pc>*{\racts\widehat{f}_A} \restore \\
\save[]+<-1.4pc,0.15pc>*{\widehat{g} \acts} \restore \widehat{Y} \ar[r]_{\widehat{\phi}}\ar[d]_{p_Y} \ar[ur]^{\sigma}&\widehat{X}\ar[u]_{\pi}\ar[d]^{p_X}\save[]+<1.4pc,0.15pc>*{\racts\widehat{f}} \restore \\
\save[]+<-1.4pc,0.15pc>*{g \acts} \restore  Y \ar[r]_{\phi} &X\save[]+<1.4pc,0.15pc>*{\racts f} \restore 
}
\]
Here, \(p_Y\) is the \(g\)-equivariant quasi-\'etale cover and \(\sigma\) is the Albanese morphism as in \cite[Theorem 1.3]{Yos21}; in particular, \(\sigma\) is a fibration of Fano type.
We take the Stein factorization of \(\phi\circ p_Y\) to obtain the small birational morphism \(\widehat{\phi}\colon \widehat{Y}\to\widehat{X}\) together with a quasi-\'etale cover \(p_X\colon\widehat{X}\to X\).
Then \(\widehat{X}\) is also klt. 
The existence of \(\widehat{f}\) follows from the universality of the Stein factorization, and \(\pi\) is the Albanese morphism of \(\widehat{X}\).
By Theorem \ref{thm-lift}, we may further replace \(\widehat{Y}\) by a \(\widehat{g}\)-equivariant maximally quasi-\'etale cover and  assume that \(\widehat{X}\) is maximally quasi-\'etale, noting that \(\widehat{\phi}\) is small.
By \cite[Theorem 1.3]{Yos21}, there exists an effective \(\mathbb{Q}\)-divisor \(\Delta\) such that \((\widehat{Y},\Delta)\) is klt and \(-(K_{\widehat{Y}}+\Delta)\) is ample over \(A\).
Since \(\widehat{\phi}\) is birational, it follows from \cite[Proposition 2.4]{Yos21} that \((\widehat{X},\widehat{\phi}_*\Delta)\) is also of Fano type over \(A\).
This proves (2).

We next prove that \(q(\widehat{X})=\widehat{q}(X)\).
Suppose, to the contrary, that this equality does not hold.
Then there exists a further quasi-\'etale cover \(X'\to\widehat{X}\) such that we have the following commutative diagram:
\[
\xymatrix{
X'\ar[d]_{p'}\ar[r]^{\pi'}&A'\ar[d]\\
\widehat{X}\ar[r]_\pi&A
}
\]
where the horizontal maps are Albanese morphisms and \(\dim A'>\dim A\).
Pick a general fibre \(F\) of \(\widehat{X}\to A\).
Take the normalization \(F'\) of an irreducible component of \(p'^{-1}(F)\).
Then the restriction \(F'\to F\) is quasi-\'etale as well. 
This implies that \(F'\) is also of Fano type, and hence rationally connected (see \cite{HM07}).
Since \(\dim A'>\dim A\), the image \(\pi'(F')\) is a positive-dimensional subvariety of \(A'\). 
This contradicts the fact that \(A'\) contains no rational curves.
This proves (1). 

Finally, we prove (3). 
By Theorem \ref{t:pos}, the tangent sheaf \(T_{\widehat{X}}\) is pseudo-effective.
Thus, by the same argument as at the beginning of the proof of Proposition \ref{prop-irreg}, we obtain that 
\[
\pi|_{\widehat{X}_{\textup{reg}}}\colon \widehat{X}_{\textup{reg}}\to A=\textup{Alb}(\widehat{X})
\]
is a submersion, and it induces the following short exact sequence:
\[
1\to \pi_1(\widehat{F}_{\textup{reg}})\to\pi_1(\widehat{X}_{\textup{reg}})\to \pi_1(A)\to 1.
\]
Since \(\widehat{F}\) is of Fano type, it follows that \(\pi_1(\widehat{F}_{\textup{reg}})\) is finite (see \cite[Theorem 2]{Bra21}; cf.~\cite[Theorem 2]{Xu14} and \cite[Theorem 1.13]{GKP16-duke}). 
Recall that a group \(G\) is called residually finite if the natural homomorphism \(G\to\widehat{G}\) to its profinite completion is injective.
Since \(\pi_1(A)\cong\mathbb{Z}^{2\dim(A)}\) is residually finite and good in the sense of \cite[Chapter I, Section 2.6]{Ser97} (see \cite[p.16, Exercise 2)(d)]{Ser97} and \cite[Proposition 3.4]{GJZZ08}), and since \(\pi_1(\widehat{F}_{\textup{reg}})\) is finite, it follows from \cite[Proposition 6.1]{GJZZ08} that \(\pi_1(\widehat{X}_{\textup{reg}})\) is also residually finite. 
In particular, we obtain an injection
\[
\pi_1(\widehat{X}_{\textup{reg}})\to \widehat{\pi}_1(\widehat{X}_{\textup{reg}}).
\]
On the other hand, since \(\pi_1(\widehat{X})\cong\pi_1(A)\) (see, e.g., \cite[Theorem 5.2]{Kol93}), we have \(\widehat{\pi}_1(\widehat{X})\cong\widehat{\pi}_1(A)\).
Together with the assumption \(\widehat{\pi}_1(\widehat{X})\cong\widehat{\pi}_1(\widehat{X}_{\textup{reg}})\), this yields an injection
\[
\pi_1(\widehat{X}_{\textup{reg}})\to \widehat{\pi}_1(A).
\]
This injection factors through \(\pi_1(A)\) by the short exact sequence induced by the submersion \(\pi|_{\widehat{X}_{\textup{reg}}}\). 
Thus we obtain an isomorphism \(\pi_1(\widehat{X}_{\textup{reg}})\cong\pi_1(A)\), which further implies that \(\pi_1(\widehat{F}_{\textup{reg}})=\{1\}\).
In particular, the regular locus of a general fibre \(\widehat{F}\) is simply connected; hence \(\widehat{F}\) has no non-trivial quasi-\'etale cover and is maximally quasi-\'etale.  
This completes the proof of Theorem \ref{t:qe-mrc} (3).
\end{proof}

\begin{remark}\label{r:cover}
Note that Yoshikawa's cover need not be maximally quasi-\'etale. 
For instance, there exists a Gorenstein del Pezzo surface of Picard number one that admits a non-isomorphic surjective endomorphism and also arises as a non-trivial quasi-\'etale quotient of the projective plane; see, e.g., \cite[Remark 1.7]{MZ19}.
\end{remark}

\section{Numerical dimension zero and the structure theorem}\label{s:num0}

In this section, we prove the existence of minimal models for compact K\"ahler klt varieties with numerical dimension zero log canonical divisor (see Theorem \ref{t:nd0}).
Then we estabilish the structure theorem for compact K\"ahler klt varieties with pseudo-effective tangent sheaf unconditionally (see Theorem \ref{t:psef-strc}).
Finally, we prove Theorem \ref{t:structure} and Corollary \ref{c:simply}.
\subsection{Termination of the MMP for the numerical dimension zero canonical divisor}
Following the spirit of \cite{Gon11}, this  subsection is devoted to the proof of the following theorem, which is a key ingredient of the proof of Theorem \ref{t:nd0}.
\begin{theorem}[{cf.~\cite[Theorem 5.2]{Gon11}}]\label{t:ter}
Let \((X,\Delta)\) be a strongly \(\mathbb{Q}\)-factorial compact K\"ahler klt pair. If \(K_X+\Delta\) is pseudo-effective and
\(\textup{nd}_X(K_X+\Delta)=0\), then every \((K_X+\Delta)\)-MMP with scaling of a K\"ahler class terminates with a log terminal model.
\end{theorem}

We shall use the notation and terminology recalled in Subsection \ref{s:nd-prelim}, together with the MMP with scaling constructed in
\cite[Proposition 4.2]{HX26} (cf.~\cite{BCHM10}).

Let us briefly recall the MMP with scaling. 
Let \((X,\Delta)\) be a strongly
\(\mathbb{Q}\)-factorial compact K\"ahler klt pair, and assume that \(K_X+\Delta\) is pseudo-effective.
Choose a K\"ahler class \(\omega\) such that \(K_X+\Delta+\omega\) is nef. 
If \(K_X+\Delta\)
is not nef, then, after rescaling \(\omega\), we may assume that its
initial nef threshold is \(\lambda_0=1\). 
The construction produces
models \((X_i,\Delta_i)\), the birational transforms \(\omega_i\) of \(\omega\), the 
non-increasing scaling numbers
\[
1=\lambda_0\geq\lambda_1\geq\cdots\geq0
\]
and, whenever \(K_{X_i}+\Delta_i\) is not nef, a
\((K_{X_i}+\Delta_i)\)-negative extremal ray \(R_i\) such that
\[
(K_{X_i}+\Delta_i+\lambda_i\omega_i)\cdot R_i=0,\]
with \(K_{X_i}+\Delta_i+\lambda_i\omega_i\) being nef.  
The pseudo-effectivity of \(K_X+\Delta\) excludes  fibre type contractions. 
Indeed, the corresponding extremal contraction is projective by \cite[Proposition 4.2]{HX26}.
On a resolution of
\(X_i\), the strict transforms of general complete-intersection curves in the fibres would form a movable curve class in \(X_i\).  
Its intersection with the
pullback of the pseudo-effective class \(K_{X_i}+\Delta_i\) is
non-negative by \cite[Proposition 18.4(b)]{Dem12}, contradicting
negativity. 
Thus all steps considered below are birational.

The following lemma plays a key role in the proof of Theorem \ref{t:ter}.

\begin{lemma}\label{l:lim}
Let \((X,\Delta)\) be a strongly \(\mathbb{Q}\)-factorial compact K\"ahler
klt pair, and let \(\omega\) be a K\"ahler class such that
\(K_X+\Delta+\omega\) is nef. Suppose that a
\((K_X+\Delta)\)-MMP with scaling of \(\omega\) has scaling numbers
\[
1=\lambda_0\geq\lambda_1\geq\cdots,
~\lambda\coloneqq\lim_i\lambda_i>0.
\]
Then the MMP terminates.
\end{lemma}

\begin{proof}
Fix a positive number \(C\) with \(0<c<\lambda\). 
Then the generalized pair
\((X,\Delta+c\overline\omega)\) is generalized klt, where
\(\overline\omega\) denotes the b-\((1,1)\)-class whose trace on \(X\)
is \(\omega\). 
Its total generalized boundary \(\Delta+c\omega\) is big
because \(c\omega\) is K\"ahler. Set \(H=(1-c)\omega\) and
\[
s_i=\frac{\lambda_i-c}{1-c}.
\]
The minimality of the nef threshold \(\lambda_i\) for \(K_{X_i}+\Delta_i\) with scaling to \(\omega_i\) shows that \(s_i\) is the nef threshold for \(K_{X_i}+\Delta_i+c\omega_i\) with scaling of \(H_i\). Moreover, on the contracted ray \(R_i\), we have 
\[
(K_{X_i}+\Delta_i+c\omega_i)\cdot R_i
=(c-\lambda_i)\omega_i\cdot R_i<0
\]
and
\[
K_{X_i}+\Delta_i+c\omega_i+s_iH_i
=K_{X_i}+\Delta_i+\lambda_i\omega_i
\]
is a nef class and is trivial on \(R_i\). 
Thus the sequence is the
\((K_X+\Delta+c\omega)\)-MMP with scaling of \(H\), which terminates by
\cite[Theorem 5.16]{HX26} because its generalized boundary is
big.
\end{proof}

\begin{proof}[Proof of Theorem \ref{t:ter}]
Fix an arbitrary \((K_X+\Delta)\)-MMP with scaling of a K\"ahler class.
Rescale the initial class as above so that \(K_X+\Delta+\omega\) is nef
and, unless \(K_X+\Delta\) is already nef, \(\lambda_0=1\).

If this MMP terminates, then pseudo-effectivity rules out a Mori fibre
space, so its end product is a log terminal model. Suppose, to the
contrary, that the sequence
\[
(X,\Delta)=(X_0,\Delta_0)
\dashrightarrow(X_1,\Delta_1)
\dashrightarrow\cdots
\]
is infinite. 
Let \(\lambda=\lim_i\lambda_i\). 
By Lemma
\ref{l:lim}, we have \(\lambda=0\).

Since each divisorial contraction descreases the dimension of the Bott-Chern cohomology, there are only finitely many divisorial contractions.
Also, flips preserve the dimension. 
So after replacing \((X,\Delta)\) with the birational model obtained after the last divisorial contraction and applying Lemma \ref{l:nd},  
we may assume that every step is a flip and that the numerical dimension still satisfies \(\textup{nd}_X(K_X+\Delta)=0\).
The new scaling numbers still tend to zero, and every \(\lambda_i\) is
positive, since \(\lambda_i=0\) would mean that \(K_{X_i}+\Delta_i\) is
nef and the MMP ends. 
By Lemma \ref{l:nd}, we have for every \(i\), 
\(\textup{nd}_{X_i}(K_{X_i}+\Delta_i)=0\).

Fix an integer \(i>0\). 
For every \(j<i\), if \(R_j\) is the ray of the \(j\)-th
flip, then
\[
(K_{X_j}+\Delta_j+\lambda_j\omega_j)\cdot R_j=0,
~
(K_{X_j}+\Delta_j)\cdot R_j<0.
\]
Thus \(\omega_j\cdot R_j>0\). 
Since \(j<i\), we have 
\[
(K_{X_j}+\Delta_j+\lambda_i\omega_j)\cdot R_j\leq0.
\]
Let \(f_j\colon X_j\to Z_j\) and
\(f_j^+\colon X_{j+1}\to Z_j\) be the flipping contraction and its flip.
On a common resolution which we denote by 
\[
r_j\colon V_j\to X_j,
~r_j^+\colon V_j\to X_{j+1}.
\]
By the negativity lemma \cite[Lemma 1.3]{Wang21}, we have 
\[
r_j^*(K_{X_j}+\Delta_j)
={r_j^+}^*(K_{X_{j+1}}+\Delta_{j+1})+F_j,
\]
with \(F_j\geq0\) being effective \(r_j^+\)-exceptional divisor. 
By the cone theorem, the threshold class
\(K_{X_j}+\Delta_j+\lambda_j\omega_j\) descends to \(Z_j\) 
\cite[Lemma 2.39 and Definition 2.17]{HX26}; hence
\[
  r_j^*(K_{X_j}+\Delta_j+\lambda_j\omega_j)
  ={r_{j}^+}^*(K_{X_{j+1}}+\Delta_{j+1}+\lambda_j\omega_{j+1}).
\]
Since \(0<\lambda_{i}\leq\lambda_j\), combining the above two equalities, we have 
\begin{align*}
&r_j^*(K_{X_j}+\Delta_j+\lambda_{i}\omega_j)=r_j^*((1-\frac{\lambda_{i}}{\lambda_j})(K_{X_j}+\Delta_j)+\frac{\lambda_{i}}{\lambda_j}(K_{X_j}+\Delta_j+\lambda_j\omega_j))\\
&=\frac{\lambda_j-\lambda_{i}}{\lambda_j}({r_j^+}^*(K_{X_{j+1}}+\Delta_{j+1})+F_j)+\frac{\lambda_{i}}{\lambda_j}{r_{j}^+}^*(K_{X_{j+1}}+\Delta_{j+1}+\lambda_j\omega_{j+1})\\
&={r_j^+}^*(K_{X_{j+1}}+\Delta_{j+1}+\lambda_{i}\omega_j)+\left(1-\frac{\lambda_{i}}{\lambda_j}\right)F_j.
\end{align*}
The last summand is effective and \(r_j^+\)-exceptional; it is zero
when \(\lambda_{i}=\lambda_j\). 
Consequently, on a common resolution
\[
p_i\colon W_i\to X,
~q_i\colon W_i\to X_i,
\]
we have
\[
p_i^*(K_X+\Delta+\lambda_i\omega)
=q_i^*(K_{X_i}+\Delta_i+\lambda_i\omega_i)+E_i,
\]
where \(E_i\) is an effective \(q_i\)-exceptional divisor given by a
weighted sum of the transforms of the \(F_j\). 
Since the composite of flips  \(X\dashrightarrow X_i\) is
an isomorphism in codimension one, every component of \(E_i\) is also \(p_i\)-exceptional; otherwise its image on
\(X\) would be a divisor contracted by the composite
\(X\dashrightarrow X_i\). 
Hence \((p_i)_*E_i=0\).

Since the threshold class \(K_{X_i}+\Delta_i+\lambda_i\omega_i\) is nef,
its pullback by \(q_i\) has zero negative part. 
Therefore, by 
\cite[Lemma A.5]{DHY23}, we have 
\[
N_{W_i}\bigl(p_i^*(K_{X_i}+\Delta_i+\lambda_i\omega_i)+E_i\bigr)=E_i.
\]
Pushing forward by \(p_i\), we obtain
\(N_X(K_X+\Delta+\lambda_i\omega)=0\). 
In other words, \(K_X+\Delta+\lambda_i\omega\) is modified nef on \(X\)
for every \(i\).

Since \(\lambda_i\to0\), the classes
\(K_X+\Delta+\lambda_i\omega\) converge to \(K_X+\Delta\). Let
\(Q\subset X\) be a prime divisor and \(\widetilde Q\) its strict
transform on a resolution \(\mu\colon\widetilde X\to X\). The coefficient
identity for the negative part gives
\[
\textup{mult}_Q N_X(K_X+\Delta)
=\textup{mult}_{\widetilde Q}
N_{\widetilde X}\bigl(\mu^*(K_X+\Delta)\bigr)
=\nu_{\widetilde X}\bigl(\mu^*(K_X+\Delta),\widetilde Q\bigr).
\]
where the last item denotes the generic minimal multiplicity (see \cite[Definition 3.1]{Bou04}). 
By lower semicontinuity of generic minimal multiplicities
\cite[Proposition 3.5]{Bou04} on the pseudo-effective cone and the nefness of \(K_X+\Delta+\lambda_i\omega\), we obtain that,
\begin{align*}
0
\leq\nu_{\widetilde X}\bigl(\mu^*(K_X+\Delta),\widetilde Q\bigr)\leq\liminf_{i\to\infty}
\nu_{\widetilde X}\bigl(\mu^*(K_X+\Delta+\lambda_i\omega),\widetilde Q\bigr)=0.
\end{align*}
The last equality follows from
\(N_X(K_X+\Delta+\lambda_i\omega)=0\) and the same coefficient identity.
Thus \(K_X+\Delta\) is modified nef. Proposition-Definition
\ref{prop-defn-num} (3) now gives \([K_X+\Delta]=0\).
In particular, \((K_X+\Delta)\cdot R_0=0\), contradicting the defining inequality \((K_X+\Delta)\cdot R_0<0\). 
Hence the MMP terminates.
\end{proof}

\subsection{Minimal models for compact K\"ahler klt varieties}
Consequent from Theorem \ref{t:ter}, we obtain the existence of the minimal model for compact K\"ahler klt varieties with numerical dimension zero log canonical divisor.
\begin{theorem}\label{t:nd0}
Let \((X,\Delta)\) be a compact K\"ahler klt pair, where \(\Delta\geq0\) is an \(\mathbb{R}\)-divisor, \(K_X+\Delta\) is \(\mathbb{R}\)-Cartier, and its Bott--Chern class is pseudo-effective with
\(\textup{nd}_X(K_X+\Delta)=0\). Then there is a small crepant morphism
\[
\nu\colon(X^{\mathrm{q}},\Delta^{\mathrm{q}})\to(X,\Delta)
\]
from a strongly \(\mathbb{Q}\)-factorial compact K\"ahler klt pair.
There exists then a finite \((K_{X^{\mathrm{q}}}+\Delta^{\mathrm{q}})\)-MMP
\[
(X^{\mathrm{q}},\Delta^{\mathrm{q}})\dashrightarrow
(X_{\min},\Delta_{\min})
\]
such that \((X_{\min},\Delta_{\min})\) is strongly \(\mathbb{Q}\)-factorial klt and
\(K_{X_{\min}}+\Delta_{\min}\equiv0\). 
If in addition, \(\Delta\) is a \(\mathbb{Q}\)-divisor and \(K_X+\Delta\) is \(\mathbb{Q}\)-Cartier, then \(K_{X_{\min}}+\Delta_{\min}\sim_{\mathbb{Q}}0\).
\end{theorem}

\begin{proof}
By \cite[Theorem 2.27]{HX26}, there is a small birational morphism
\[
\nu\colon X'\to X
\]
such that \(X'\) is compact K\"ahler and strongly
\(\mathbb{Q}\)-factorial. Let \(\Delta'=\nu_*^{-1}\Delta\). Since \(\nu\)
is small and \(K_X+\Delta\) is \(\mathbb{R}\)-Cartier, we have
\[
K_{X'}+\Delta'=\nu^*(K_X+\Delta).
\]
Thus \(\nu\) is log crepant, \((X',\Delta')\) is klt, and
\(\Delta'\geq0\). If \(r\colon W\to X'\) is a resolution, then
\[
r^*(K_{X'}+\Delta')=(\nu\circ r)^*(K_X+\Delta),
\]
so the numerical dimension remains zero by Proposition-Definition
\ref{prop-defn-num}.

By \cite[Proposition 4.2]{HX26}, we may run a \((K_{X'}+\Delta')\)-MMP with scaling, and every intermediate model is still a strongly \(\mathbb{Q}\)-factorial compact K\"ahler klt pair.
By Theorem \ref{t:ter}, this MMP terminates after finitely many steps and ends with a strongly
\(\mathbb{Q}\)-factorial compact K\"ahler klt pair
\((X_{\min},\Delta_{\min})\) for which 
\(K_{X_{\min}}+\Delta_{\min}\) is nef. 
By Lemma
\ref{l:nd},
\[
\textup{nd}_{X_{\min}}(K_{X_{\min}}+\Delta_{\min})=0.
\]
Recall that klt varieties having rational singularities.
The end product \(K_{X_{\min}}+\Delta_{\min}\) is nef and thus modified nef, so that Proposition-Definition
\ref{prop-defn-num}(3) gives
\([K_{X_{\min}}+\Delta_{\min}]
=0\) 
in 
\(H^{1,1}_{\mathrm{BC}}(X_{\min},\mathbb{R})\)

Suppose now that \(\Delta\) is a \(\mathbb{Q}\)-divisor with \(K_X+\Delta\) \(\mathbb{Q}\)-Cartier.
Both properties are preserved under crepant pullback and throughout the MMP, so \(\Delta_{\min}\) is rational and \(K_{X_{\min}}+\Delta_{\min}\) is \(\mathbb{Q}\)-Cartier as well.
By \cite[Corollary 1.18]{CGP23}, there exist a sufficiently divisible integer \(m>0\) and an effective Cartier divisor \(G\in |m(K_{X_{\min}}+\Delta_{\min})|\).
Intersecting it with a K\"ahler class \((\dim X-1)\) times, we obtain that \(G\sim 0\).
This proves the asserted \(\mathbb{Q}\)-linear triviality.
\end{proof}

\subsection{MRC fibrations for pseudo-effective tangent sheaves, Proofs of Theorem \ref{t:structure} and Corollary \ref{c:simply}}
The proceeding Theorem \ref{t:nd0} for klt pairs resolves the only conjectural part required for the structure theorem in \cite{MZ-TJM}. 
Therefore, we are now able to obtain the structure theorem below, which generalizes all the previous works \cite{DPS94,HIM22,IMZ23,MQ25}. 
\begin{theorem}\label{t:psef-strc}
Let \(X\) be a compact klt K\"ahler variety with pseudo-effective tangent sheaf.
After a finite quasi-\'etale cover, there is a holomorphic MRC fibration
\(\pi\colon X\to T\)
onto a complex torus such that a very general fibre is a rationally connected
klt projective variety with pseudo-effective tangent sheaf.
\end{theorem}

\begin{proof}
Applying Theorem \ref{t:nd0} with \(\Delta=0\) shows that every compact K\"ahler manifold \(Y\) with \(\textup{nd}_Y(K_Y)=0\) is bimeromorphic to a compact klt K\"ahler \(Y_{\min}\) with \(K_{Y_{\min}}\equiv 0\).
This establishes 
\cite[Conjecture~1.7]{MZ-TJM}. 
Consequently, \cite[Theorem 1.6]{MZ-TJM} applies 
unconditionally, and yields the theorem.
\end{proof}

\begin{proof}[Proof of Theorem \ref{t:structure}]
By Theorem \ref{t:pos}, \(T_X\) is pseudo-effective. 
Then it follows from  Theorem \ref{t:psef-strc} that there is a maximally quasi-\'etale cover \(X'\to X\) such that \(X'\) admits a holomorphic MRC fibration onto a complex torus. 

By Theorem \ref{thm-lift} and its proof, there exists a further quasi-\'etale (and hence \'etale) cover \(\widehat{X}\to X'\) of \(X'\)  such that \(f\) lifts to a surjective endomorphism \(\widehat{f}\colon\widehat{X}\to\widehat{X}\).
Then \(\widehat{f}\) is int-amplified as well (see Proposition \ref{p:int}).
Taking the Stein factorization of \(\widehat{X}\to X\to T\), we obtain a surjective holomorphic map \(\widehat{X}\to\widehat{T}\) followed by a finite surjective morphism \(\widehat{T}\to T\). 
Since \(\widehat{X}\to X\) is \'etale, so it \(\widehat{T}\to T\); in particular, \(\widehat{T}\) is a complex torus as well.
Together with \(q(\widehat{X})=q(X)=\dim T\), and \(T_{\widehat{X}}\) being also pseudo-effective, we obtain that \(\widehat{X}\to\widehat{T}\) is the holomorphic MRC fibration.

In what follows, with \(X\) replaced by \(\widehat{X}\), we shall consider the holomorphic MRC fibration \(\pi\colon X\to T\) onto a complex torus, which is also the Albanese morphism.
Then \(f\) descends to \(g\colon T\to T\) by the universal property, which is also int-amplified.
By \cite[Proof of Lemma 4.3]{Men20}, the closed subset \(S\) of \(T\) over which \(\alpha\) has reducible, or non-reduced, or non-equidimensional fibers is \(g^{-1}\)-invariant.
Therefore, \cite[Lemma 5.4]{Zho21-Asian} implies that 
\(S\) is empty. 
Consequently, \(\pi\) has irreducible and reduced fibres, and is equidimensional and hence flat by the miracle flatness. 
By the flatness of \(\pi\) and the upper semicontinuity of \(h^i(X_t,\mathcal{O}_{X_t})\), 
we obtain that a general fibre satisfies \(h^i(X_t,\mathcal{O}_{X_t})=0\) for every \(i\geq 1\).
Hence, it follows from \cite[Theorem 4.1 and Corollary 4.2]{CH24} that \(\pi\) is projective. 

Finally, if \(X\) is smooth, then by applying \cite{MQ25}, the MRC fibration can be chosen to be smooth so that every fibre of \(\pi\) is smooth rationally connected.
By choosing any periodic point, which is dense in \(S\) (see \cite{DZ23}), we obtain that the associated periodic fibre is of Fano type by \cite[Corollary 1.4]{Yos21}. 
We finish the proof of the theorem.
\end{proof}

\begin{proof}[Proof of Corollary \ref{c:simply}]
First, by Theorem \ref{t:structure}, there exists an \(f\)-equivariant maximally quasi-\'etale cover \( \widehat{X}\to X\) such that \(\widehat{X}\) admits a holomorphic MRC fibration onto a complex torus \(T\).

\((1)\Rightarrow (2)\)  is trivial.
For 
\((2)\Rightarrow (3)\), assuming (2), we have \(q(\widehat{X})=0\) so that \(\dim T=0\) and \(\widehat{X}\) is rationally connected; consequently \(\widehat{X}\) has trivial fundamental group.

For \((3)\Rightarrow (1)\), 
assume that \(\pi_1(\widehat{X})\) is finite. 
Then by taking the resolution and applying \cite[Theorem]{Tak03} and \cite[Theorem 5.2]{Kol93}, we have \(\pi_1(\widehat{X})\cong\pi_1(T)\) is trivial.
Therefore, \(\widehat{X}\) is maximally quasi-\'etale and rationally connected.
Hence, by \cite{Yos21}, we see that it is of Fano type.
Then \(X\), as the finite quasi-\'etale quotient of \(\widehat{X}\), is of Fano type as well.

We finish the proof of the corollary.
\end{proof}

\section{Examples, Proof of Theorem \ref{thm-example}}\label{s:ex}
In this section, we collect several examples illustrating our main results and prove Theorem \ref{thm-example}. 
We first consider certain   manifolds \(X\) for which \(\mathcal{O}_{\mathbb{P}(T_X)}(1)\) is pseudo-effective, while the tangent bundle itself is not pseudo-effective in the sense of Definition \ref{df-pos}. 

\begin{example}\label{ex-ruled} 
Let \(X=\mathbb{P}_C(\mathcal{E})\) be a ruled surface over a curve \(C\) of genus \(g(C)\). Then \(\mathcal{O}_{\mathbb{P}(T_X)}(1)\) is always pseudo-effective (see, e.g., \cite[Proposition 6.1]{JLZ25}). However, the tangent bundle \(T_X\) is pseudo-effective if and only if \(g(C)\leq 1\) (cf.~\cite[Proposition 4.2]{HIM22}). Moreover \(X\) admits an int-amplified endomorphism if and only if \(g(C)\leq 1\) and, after a possible \'etale base change, the vector bundle \(\mathcal{E}\) splits into a direct sum of line bundles; see \cite[Proposition 2.3.1]{ZhaSW06} and \cite[Section 7]{MY21}. 
\end{example}

To provide more interesting examples—like rationally connected manifolds—we construct the following one, which is based on a discussion with Masataka Iwai.
\begin{example}\label{ex-psef}
Fix any \(n\in\mathbb{N}\).
Let \(X\) be the blow-up of \(Y\coloneqq Y_1\times Y_2=\mathbb{P}^1\times\mathbb{P}^1\) at \(n\) distinct points lying in the same fibre of the first projection.
It follows from \cite[Proposition 4.5]{HIM22} that such a surface has a non-pseudo-effective tangent bundle whenever \(n\geq 5\).
Note that the ``general position'' assumption in \cite[Proposition 4.5]{HIM22} is used only to ensure that the projection of the blown-up points in a certain direction contains more than \(4\) points.  
In the present example, the blown-up points project to \(n\) distinct points via the second projection \(p_2\).
Thus we can apply \cite[Proposition 4.5]{HIM22} to conclude that \(T_X\) is not pseudo-effective. 
Moreover, such a surface admits an int-amplified endomorphism if and only if it is toric (see \cite{Nak02}), which is equivalent to \(n\leq 2\).

However, we can show that \(H^0(X,T_X)\neq 0\) and thus \(\mathcal{O}_{\mathbb{P}(T_X)}(1)\) is pseudo-effective.
We consider the following diagram:
\[
\xymatrix{ 
X\ar[r]^{\pi}&Y\ar[r]^{p_2}\ar[d]_{p_1}&Y_2\\
&Y_1&
}
\]
Suppose that the \(n\) distinct points lie in the same fibre \(\ell_1\) of the first projection \(p_1\).
Let \(\ell_2\) be any fibre of \(p_2\) which does not contain any of the \(n\) blown-up points.
Let \(E_1,\cdots,E_n\) be the \(\pi\)-exceptional divisors.
Then we have 
\[
K_X=\pi^*K_Y+\sum_{i=1}^n E_i.
\]
We note that \(T_Y\cong p_1^*\mathcal{O}_{Y_1}(2)\oplus p_2^*\mathcal{O}_{Y_2}(2)\cong \mathcal{O}_Y(2\ell_1)\oplus\mathcal{O}_Y(2\ell_2)\).
We also have the natural injection
\[
0\to \pi^*\Omega_Y\to\Omega_X.
\]
Hence we obtain
\[
0\to \pi^*T_Y(K_Y)\to T_X(K_X).
\]
This implies that
\[
0\to (\pi^*\mathcal{O}_{Y}(2\ell_1)\oplus \pi^*\mathcal{O}_Y(2\ell_2))\otimes \mathcal{O}_X(-\sum E_i)\to T_X.
\]
Since we have 
$$
\pi^*\mathcal{O}_Y(2\ell_1)\otimes \mathcal{O}_X(-\sum E_i)=\mathcal{O}_X(2\ell_1'+\sum E_i),$$
where \(\ell_1'\) is the proper transform of \(\ell_1\), we obtain \(H^0(X,T_X)\neq 0\).
\end{example}

In what follows, we give an example showing that the maximally quasi-\'etale assumption in Corollary \ref{c:simply} (3)  cannot be removed (cf.~Theorem \ref{t:psef-rc}).  
\begin{example}\label{ex-mqe}
Let \(A\) be an abelian surface and let \(X\) be the quotient of \(A\) by the involution. 
Then \(X\) has du Val singularities (indeed, 16 singularities of type \(A_1\)), and its minimal resolution \(\widetilde{X}\) is a K3 surface whose tangent bundle is not pseudo-effective (see, e.g., \cite[Theorem 1.6]{HP19}). 
By \cite[Theorem 1.1]{Tak03}, \(X\) is simply connected, and by \cite[Theorem 1.1]{Shi25}, \(X\) always admits an int-amplified endomorphism. 
However, \(X\) is not of Fano type, and such an int-amplified endomorphism cannot be lifted to its minimal resolution.
\end{example}

The Campana-Peternell conjecture \cite{CP91} asserts that  Fano manifolds with nef tangent bundles are rational homogeneous, in other words, they are quotients of semi-simple linear algebraic
groups by parabolic subgroups.
We may ask whether a rationally connected projective manifold with pseudo-effective tangent bundle is almost homogeneous, i.e., there exists an open dense \(\textup{Aut}^0(X)\)-orbit. 
Nevertheless, the following example shows that this is not the case.

\begin{example}\label{e:dp}
Let \(X\) be a del Pezzo manifold of degree four, that is, a complete intersection of two smooth quadric hypersurfaces.
When \(\dim(X)\geq3\), \(X\) has Picard number one. 
By \cite[Theorem 1.1 (a)]{BEHLV24},
we have \(H^0(X,\textup{Sym}^2T_X)\neq 0\) and hence the tautological line bundle \(\mathcal{O}_{\mathbb{P}(T_X)}(1)\) is pseudo-effective; further, its non-nef locus, which is contained in the stable base locus, does not dominate the base (see \cite[Proposition 3.1]{BEHLV24}).
Therefore, \(X\) has pseudo-effective tangent bundle. 
However, since \(H^0(X,T_X)=0\) and hence \(X\) has trivial connected component \(\textup{Aut}^0(X)\) of the automorphism group, 
\(X\) is not almost homogeneous.
\end{example}

Motivated by the dynamical perspective of \cite[Corollary 1.4]{Yos21}, 
we may ask whether every rationally connected projective manifold with pseudo-effective tangent bundle is of Fano type. However, the following example shows that the answer is negative, either.

\begin{example}\label{e:non-ft}
Let \(X\cong\mathbb{P}^3\) be the projective 3-space with the homogeneous  coordinates \([x_0:x_1:x_2:x_3]\).
Let \(H\) be the hyperplane section of \(X\) defined by \(\{x_0=0\}\). 
Then \(X\) admits a \(\mathbb{C}^*\)-action on \(X\) given by 
\[
t\cdot[x_0\colon x_1\colon x_2\colon x_3]=[x_0\colon tx_1\colon tx_2\colon tx_3]
\]
which fixes \(H\) pointwise.
Let \(S\) be any finite set of isolated points on \(H\) and denote by \(W_X\) (resp.~\(W_H\)) the blow-up of \(X\) (resp.~\(H\)) along \(S\).
Then it follows from \cite[Example 2.2 and Proposition 2.3]{FH20} that \(W_X\) is an almost homogeneous variety, and hence \(T_{W_X}\) is pseudo-effective, noting that there is a generically surjective morphism \(\mathcal{O}_{W_X}^{\oplus r}\to T_{W_X}\) for some \(r=h^0(W_X,T_{W_X})\).
More precisely, in this case, \(W_X\) admits a natural \(\mathbb{G}_a^3\)-action induced by the action on \(X\) as follows:
for each \((s_1,s_2,s_3)\in\mathbb{G}_a^3\), the action of \((s_1,s_2,s_3)\) on \(X\) is defined by
\[
(s_1,s_2,s_3).[x_0:x_1:x_2:x_3]=[x_0:x_1+s_1x_0:x_2+s_2x_0:x_3+s_3x_0].
\]
These actions fix \(H\) pointwise, and there is a unique open orbit of \(\mathbb{G}_a^3\colon W_X\to W_X\)  isomorphic to \(X\backslash H\).

On the other hand, if we choose a general finite set \(S\subseteq H\)  not contained in a single line, it follows from \cite[Proposition 6.1]{CLO16} that the Mori cone \(\overline{\textup{NE}}(W_X)\) of \(W_X\) coincides with the Mori cone \(\overline{\textup{NE}}(W_H)\) of \(W_H\).  
Since the latter contains infinitely many \((-1)\)-curves when \(S\) is chosen to contain more than 9 general points in \(H\), it follows that \(\overline{\textup{NE}}(W_X)\) is not finitely generated, either. 
In particular, by \cite[Corollary 1.3.2]{BCHM10}, such \(W_X\) is not of Fano type.
\end{example}

Let us conclude the paper with the following remark. 

\begin{remark}\label{r:pos-cur}
Besides the pseudo-effectivity as verified in Theorems \ref{t:pos} and \ref{t:pos-refine}, we may ask whether the tangent sheaf $T_{X}$ satisfies other  positivity properties when the variety \(X\) admits an int-amplified endomorphism, for example, whether $T_{X}$ is positively curved or almost nef (see, e.g., \cite[Definition 19.1]{HPS18}, \cite[Definition 2.2.2]{PT18}, and \cite[Definition 2.1]{IMZ23}); however, neither property holds in general.
\begin{enumerate}
\item By \cite[Proposition 2.3.1]{ZhaSW06}, the ruled surface \(\mathbb{P}(\mathcal{O}\oplus\mathcal{O}(np))\) over an elliptic curve admits an int-amplified endomorphism for any \(n\in\mathbb{N}\), but unless \(n=0\), its tangent bundle is not positively curved (see \cite[Proposition 4.2]{HIM22}).
In particular, \(\mathbb{P}(\mathcal{O}\oplus\mathcal{O}(np))\) is weakly positively curved but not positively curved when \(n\geq 1\).
\item By \cite[Theorem 1.1]{Shi25}, any quasi-\'etale quotient of an abelian variety admits an int-amplified endomorphism, but unless the quotient is smooth, its tangent sheaf is never almost nef (see \cite[Corollary 1.3]{IMZ23}; cf.~\cite[Example 6.5]{IMZ23}). 
\item A rationally connected projective manifold admitting an int-amplified endomorphism is conjectured to be toric (see \cite[Conjecture 4.4]{Fak03}).
Hence, in addition to Conjecture \ref{Conj-big}, it is still  reasonable to study these positivity properties for rationally connected projective manifolds. 
\end{enumerate}
\end{remark}

\bibliographystyle{amsalpha}

\bibliography{bib-ref}

@misc {Mat25,
    AUTHOR = {Matsumura, Shin-ichi},
     TITLE = {Fundamental groups of compact {K}\"ahler manifolds with
              semi-positive holomorphic sectional curvature},
      YEAR = {2025},
      NOTE = {Preprint, arXiv:2502.00367},
    EPRINT = {2502.00367},
ARCHIVEPREFIX = {arXiv},
PRIMARYCLASS = {math.DG},
       URL = {https://arxiv.org/abs/2502.00367},
}

@article {Bea01,
    AUTHOR = {Beauville, Arnaud},
     TITLE = {Endomorphisms of hypersurfaces and other manifolds},
   JOURNAL = {Internat. Math. Res. Notices},
  FJOURNAL = {International Mathematics Research Notices},
      YEAR = {2001},
    NUMBER = {1},
     PAGES = {53--58},
      ISSN = {1073-7928,1687-0247},
   MRCLASS = {14J70},
  MRNUMBER = {1809497},
MRREVIEWER = {S\'andor\ J.\ Kov\'acs},
       DOI = {10.1155/S1073792801000034},
       URL = {https://doi.org/10.1155/S1073792801000034},
}

@article {BEHLV24,
    AUTHOR = {Beauville, Arnaud and Etesse, Antoine and H\"oring, Andreas
              and Liu, Jie and Voisin, Claire},
     TITLE = {Symmetric tensors on the intersection of two quadrics and
              {L}agrangian fibration},
   JOURNAL = {Moduli},
  FJOURNAL = {Moduli},
    VOLUME = {1},
      YEAR = {2024},
     PAGES = {Paper No. e4, 19},
      ISSN = {2949-7647,2977-1382},
  MRCLASS = {70H06 (14J45)},
  MRNUMBER = {4892590},
MRREVIEWER = {Tetsuya\ Taniguchi},
       DOI = {10.1112/mod.2024.3},
       URL = {https://doi.org/10.1112/mod.2024.3},
}

@article {BCHM10,
    AUTHOR = {Birkar, Caucher and Cascini, Paolo and Hacon, Christopher D.
              and McKernan, James},
     TITLE = {Existence of minimal models for varieties of log general type},
   JOURNAL = {J. Amer. Math. Soc.},
  FJOURNAL = {Journal of the American Mathematical Society},
    VOLUME = {23},
      YEAR = {2010},
    NUMBER = {2},
     PAGES = {405--468},
      ISSN = {0894-0347,1088-6834},
   MRCLASS = {14E30 (14E05)},
  MRNUMBER = {2601039},
MRREVIEWER = {Mark\ Gross},
       DOI = {10.1090/S0894-0347-09-00649-3},
       URL = {https://doi.org/10.1090/S0894-0347-09-00649-3},
}

@phdthesis {Bou02,
    AUTHOR = {Boucksom, S{\'e}bastien},
     TITLE = {C{\^o}nes positifs des vari{\'e}t{\'e}s complexes compactes},
    SCHOOL = {Universit{\'e} Joseph Fourier, Grenoble I},
      YEAR = {2002},
       URL = {https://theses.hal.science/tel-00002268},
}

@article {Bou04,
    AUTHOR = {Boucksom, S\'{e}bastien},
     TITLE = {Divisorial {Z}ariski decompositions on compact complex
              manifolds},
   JOURNAL = {Ann. Sci. \'{E}cole Norm. Sup. (4)},
  FJOURNAL = {Annales Scientifiques de l'\'{E}cole Normale Sup\'{e}rieure. Quatri\`eme
              S\'{e}rie},
    VOLUME = {37},
      YEAR = {2004},
    NUMBER = {1},
     PAGES = {45--76},
      ISSN = {0012-9593},
   MRCLASS = {32J18 (32C30)},
  MRNUMBER = {2050205},
MRREVIEWER = {Adam Gregory Harris},
       DOI = {10.1016/j.ansens.2003.04.002},
       URL = {https://doi.org/10.1016/j.ansens.2003.04.002},
}

@article {BEGZ10,
    AUTHOR = {Boucksom, S{\'e}bastien and Eyssidieux, Philippe and Guedj,
              Vincent and Zeriahi, Ahmed},
     TITLE = {Monge--Amp{\`e}re equations in big cohomology classes},
   JOURNAL = {Acta Math.},
  FJOURNAL = {Acta Mathematica},
    VOLUME = {205},
      YEAR = {2010},
    NUMBER = {2},
     PAGES = {199--262},
      ISSN = {0001-5962,1871-2509},
   MRCLASS = {32U40 (32Q20 32U15 32W20)},
  MRNUMBER = {2746347},
MRREVIEWER = {S{\l}awomir\ Dinew},
       DOI = {10.1007/s11511-010-0054-7},
       URL = {https://doi.org/10.1007/s11511-010-0054-7},
}

@incollection {BG13,
    AUTHOR = {Boucksom, S{\'e}bastien and Guedj, Vincent},
     TITLE = {Regularizing properties of the {K}\"ahler--{R}icci flow},
 BOOKTITLE = {An introduction to the {K}\"ahler--{R}icci flow},
    SERIES = {Lecture Notes in Math.},
    VOLUME = {2086},
     PAGES = {189--237},
 PUBLISHER = {Springer, Cham},
      YEAR = {2013},
      ISBN = {978-3-319-00818-9,978-3-319-00819-6},
   MRCLASS = {32W20 (14E99 53C44)},
  MRNUMBER = {3185334},
MRREVIEWER = {S{\l}awomir\ Ko{\l}odziej},
       DOI = {10.1007/978-3-319-00819-6_4},
       URL = {https://doi.org/10.1007/978-3-319-00819-6_4},
}

@article {BdFF12,
    AUTHOR = {Boucksom, Sebastien and de Fernex, Tommaso and Favre, Charles},
     TITLE = {The volume of an isolated singularity},
   JOURNAL = {Duke Math. J.},
  FJOURNAL = {Duke Mathematical Journal},
    VOLUME = {161},
      YEAR = {2012},
    NUMBER = {8},
     PAGES = {1455--1520},
      ISSN = {0012-7094,1547-7398},
   MRCLASS = {14J17 (14B05 14C20 14F18)},
  MRNUMBER = {2931273},
MRREVIEWER = {Cihan\ \"Ozg\"ur},
       DOI = {10.1215/00127094-1593317},
       URL = {https://doi.org/10.1215/00127094-1593317},
}

@article {DP04,
    AUTHOR = {Demailly, Jean-Pierre and Paun, Mihai},
     TITLE = {Numerical characterization of the {K}\"ahler cone of a
              compact {K}\"ahler manifold},
   JOURNAL = {Ann. of Math. (2)},
  FJOURNAL = {Annals of Mathematics. Second Series},
    VOLUME = {159},
      YEAR = {2004},
    NUMBER = {3},
     PAGES = {1247--1274},
      ISSN = {0003-486X,1939-8980},
   MRCLASS = {32J27 (32Q15)},
  MRNUMBER = {2113021},
MRREVIEWER = {Philippe\ P.\ Eyssidieux},
       DOI = {10.4007/annals.2004.159.1247},
       URL = {https://doi.org/10.4007/annals.2004.159.1247},
}

@article {Bra21,
    AUTHOR = {Braun, Lukas},
     TITLE = {The local fundamental group of a {K}awamata log terminal
              singularity is finite},
   JOURNAL = {Invent. Math.},
  FJOURNAL = {Inventiones Mathematicae},
    VOLUME = {226},
      YEAR = {2021},
    NUMBER = {3},
     PAGES = {845--896},
      ISSN = {0020-9910,1432-1297},
   MRCLASS = {14F35 (14B05 14J45 32S50)},
  MRNUMBER = {4337973},
       DOI = {10.1007/s00222-021-01062-0},
       URL = {https://doi.org/10.1007/s00222-021-01062-0},
}

@article {Cam04,
    AUTHOR = {Campana, Fr\'ed\'eric},
     TITLE = {Orbifolds, special varieties and classification theory},
   JOURNAL = {Ann. Inst. Fourier (Grenoble)},
  FJOURNAL = {Universit\'e{} de Grenoble. Annales de l'Institut Fourier},
    VOLUME = {54},
      YEAR = {2004},
    NUMBER = {3},
     PAGES = {499--630},
      ISSN = {0373-0956,1777-5310},
   MRCLASS = {14E05 (14D06 14J40 32Q57 35Q15)},
  MRNUMBER = {2097416},
MRREVIEWER = {Dan\ Abramovich},
       DOI = {10.5802/aif.2027},
       URL = {https://doi.org/10.5802/aif.2027},
}

@article {CP91,
    AUTHOR = {Campana, Fr\'{e}d\'{e}ric and Peternell, Thomas},
     TITLE = {Projective manifolds whose tangent bundles are numerically
              effective},
   JOURNAL = {Math. Ann.},
  FJOURNAL = {Mathematische Annalen},
    VOLUME = {289},
      YEAR = {1991},
    NUMBER = {1},
     PAGES = {169--187},
      ISSN = {0025-5831},
   MRCLASS = {14J30 (14E20)},
  MRNUMBER = {1087244},
MRREVIEWER = {I. Dolgachev},
       DOI = {10.1007/BF01446566},
       URL = {https://doi.org/10.1007/BF01446566},
}

@misc {CDM26,
    AUTHOR = {Cao, Junyan and Deng, Ya and Matsumura, Shin-ichi},
     TITLE = {A flatness criterion for pseudo-effective sheaves on compact
              {K}\"ahler spaces},
      YEAR = {2026},
       NOTE = {Preprint, arXiv:2609.05154},
    EPRINT = {2609.05154},
ARCHIVEPREFIX = {arXiv},
PRIMARYCLASS = {math.AG},
       URL = {https://arxiv.org/abs/2609.05154},
}

@article {CGP23,
    AUTHOR = {Cao, Junyan and Guenancia, Henri and P\u aun, Mihai},
     TITLE = {Variation of singular {K}\"ahler-{E}instein metrics: {K}odaira
              dimension zero},
      NOTE = {With an appendix by Valentino Tosatti},
   JOURNAL = {J. Eur. Math. Soc. (JEMS)},
  FJOURNAL = {Journal of the European Mathematical Society (JEMS)},
    VOLUME = {25},
      YEAR = {2023},
    NUMBER = {2},
     PAGES = {633--679},
      ISSN = {1435-9855,1435-9863},
   MRCLASS = {14J10 (14E30 14J32 32Q20)},
  MRNUMBER = {4556792},
MRREVIEWER = {Ruadha\'i\ Dervan},
       DOI = {10.4171/jems/1184},
       URL = {https://doi.org/10.4171/jems/1184},
}

@misc {CZ25,
    AUTHOR = {Chang, Wentao and Zhang, De-Qi},
     TITLE = {Log {C}alabi--{Y}au structure of algebaic varieties admitting
              a polarized endomorphism},
      YEAR = {2025},
      NOTE = {Preprint, arXiv:2509.17927},
    EPRINT = {2509.17927},
ARCHIVEPREFIX = {arXiv},
PRIMARYCLASS = {math.AG},
       URL = {https://arxiv.org/abs/2509.17927},
}

@misc {CH24,
    AUTHOR = {Claudon, Beno{\^\i}t and H{\"o}ring, Andreas},
     TITLE = {Projectivity criteria for {K}\"ahler morphisms},
      YEAR = {2024},
      NOTE = {Preprint, arXiv:2404.13927},
    EPRINT = {2404.13927},
ARCHIVEPREFIX = {arXiv},
PRIMARYCLASS = {math.AG},
       URL = {https://arxiv.org/abs/2404.13927},
}

@article {CLO16,
    AUTHOR = {Coskun, Izzet and Lesieutre, John and Ottem, John Christian},
     TITLE = {Effective cones of cycles on blowups of projective space},
   JOURNAL = {Algebra Number Theory},
  FJOURNAL = {Algebra \& Number Theory},
    VOLUME = {10},
      YEAR = {2016},
    NUMBER = {9},
     PAGES = {1983--2014},
      ISSN = {1937-0652,1944-7833},
   MRCLASS = {14C25 (14E07 14E30 14M07)},
  MRNUMBER = {3576118},
MRREVIEWER = {Fumio\ Hazama},
       DOI = {10.2140/ant.2016.10.1983},
       URL = {https://doi.org/10.2140/ant.2016.10.1983},
}

@misc {DHY23,
    AUTHOR = {Das, Omprokash and Hacon, Christopher and Y{\'a}{\~n}ez,
              Jos{\'e} Ignacio},
     TITLE = {{MMP} for generalized pairs on {K}\"ahler 3-folds},
      YEAR = {2023},
      NOTE = {To appear, arXiv:2305.00524},
    EPRINT = {2305.00524},
ARCHIVEPREFIX = {arXiv},
PRIMARYCLASS = {math.AG},
       URL = {https://arxiv.org/abs/2305.00524},
}

@article {DPS94,
    AUTHOR = {Demailly, Jean-Pierre and Peternell, Thomas and Schneider,
              Michael},
     TITLE = {Compact complex manifolds with numerically effective tangent
              bundles},
   JOURNAL = {J. Algebraic Geom.},
  FJOURNAL = {Journal of Algebraic Geometry},
    VOLUME = {3},
      YEAR = {1994},
    NUMBER = {2},
     PAGES = {295--345},
      ISSN = {1056-3911},
   MRCLASS = {32J27 (14J45 32L07)},
  MRNUMBER = {1257325},
MRREVIEWER = {Yoichi Miyaoka},
}

@article {Dem85,
    AUTHOR = {Demailly, Jean-Pierre},
     TITLE = {Mesures de {M}onge-{A}mp\`ere et caract\'erisation
              g\'eom\'etrique des vari\'et\'es alg\'ebriques affines},
   JOURNAL = {M\'em. Soc. Math. France (N.S.)},
  FJOURNAL = {M\'emoires de la Soci\'et\'e{} Math\'ematique de France.
              Nouvelle S\'erie},
    NUMBER = {19},
      YEAR = {1985},
     PAGES = {124},
      ISSN = {0037-9484},
   MRCLASS = {32H35 (32C10 32F05)},
  MRNUMBER = {813252},
MRREVIEWER = {G.\ M.\ Khenkin},
}

@book {Dem12,
    AUTHOR = {Demailly, Jean-Pierre},
     TITLE = {Analytic methods in algebraic geometry},
    SERIES = {Surveys of Modern Mathematics},
    VOLUME = {1},
 PUBLISHER = {International Press, Somerville, MA; Higher Education Press,
              Beijing},
      YEAR = {2012},
     PAGES = {viii+231},
      ISBN = {978-1-57146-234-3},
   MRCLASS = {32-02 (14C30 14F18 32J25 32Q15 32U40)},
  MRNUMBER = {2978333},
MRREVIEWER = {Valentino Tosatti},
}

@article {DS04-Ens,
    AUTHOR = {Dinh, Tien-Cuong and Sibony, Nessim},
     TITLE = {Regularization of currents and entropy},
   JOURNAL = {Ann. Sci. \'Ecole Norm. Sup. (4)},
  FJOURNAL = {Annales Scientifiques de l'\'Ecole Normale Sup\'erieure.
              Quatri\`eme S\'erie},
    VOLUME = {37},
      YEAR = {2004},
    NUMBER = {6},
     PAGES = {959--971},
      ISSN = {0012-9593},
   MRCLASS = {32U40 (32C30 32H04 32Q15 37B40)},
  MRNUMBER = {2119243},
MRREVIEWER = {Ma\l gorzata\ Stawiska},
       DOI = {10.1016/j.ansens.2004.09.002},
       URL = {https://doi.org/10.1016/j.ansens.2004.09.002},
}

@article {DS05,
    AUTHOR = {Dinh, Tien-Cuong and Sibony, Nessim},
     TITLE = {Une borne sup\'erieure pour l'entropie topologique d'une
              application rationnelle},
   JOURNAL = {Ann. of Math. (2)},
  FJOURNAL = {Annals of Mathematics. Second Series},
    VOLUME = {161},
      YEAR = {2005},
    NUMBER = {3},
     PAGES = {1637--1644},
      ISSN = {0003-486X,1939-8980},
   MRCLASS = {32H50 (37B40 37F05)},
  MRNUMBER = {2180409},
MRREVIEWER = {Jean-Yves\ Briend},
       DOI = {10.4007/annals.2005.161.1637},
       URL = {https://doi.org/10.4007/annals.2005.161.1637},
}

@article {DZ23,
    AUTHOR = {Dinh, Tien-Cuong and Zhong, Guolei},
     TITLE = {Periodic points for meromorphic self-maps of {F}ujiki
              varieties},
   JOURNAL = {Indiana Univ. Math. J.},
  FJOURNAL = {Indiana University Mathematics Journal},
    VOLUME = {74},
      YEAR = {2025},
    NUMBER = {6},
     PAGES = {1561--1587},
      ISSN = {0022-2518,1943-5258},
   MRCLASS = {37F80 (32H04 32H50 37C35)},
  MRNUMBER = {5011964},
       DOI = {10.1512/iumj.2025.74.60572},
       URL = {https://doi.org/10.1512/iumj.2025.74.60572},
}

@article {Fak03,
    AUTHOR = {Fakhruddin, Najmuddin},
     TITLE = {Questions on self maps of algebraic varieties},
   JOURNAL = {J. Ramanujan Math. Soc.},
  FJOURNAL = {Journal of the Ramanujan Mathematical Society},
    VOLUME = {18},
      YEAR = {2003},
    NUMBER = {2},
     PAGES = {109--122},
      ISSN = {0970-1249,2320-3110},
   MRCLASS = {14G05 (37A45)},
  MRNUMBER = {1995861},
MRREVIEWER = {Yuri\ Tschinkel},
}

@article {FH20,
    AUTHOR = {Fu, Baohua and Hwang, Jun-Muk},
     TITLE = {Euler-symmetric projective varieties},
   JOURNAL = {Algebr. Geom.},
  FJOURNAL = {Algebraic Geometry},
    VOLUME = {7},
      YEAR = {2020},
    NUMBER = {3},
     PAGES = {377--389},
      ISSN = {2313-1691,2214-2584},
   MRCLASS = {14M27},
  MRNUMBER = {4087864},
MRREVIEWER = {Hossein\ Sabzrou},
       DOI = {10.14231/ag-2020-011},
       URL = {https://doi.org/10.14231/ag-2020-011},
}

@article {Gon11,
    AUTHOR = {Gongyo, Yoshinori},
     TITLE = {On the minimal model theory for dlt pairs of numerical log
              {K}odaira dimension zero},
   JOURNAL = {Math. Res. Lett.},
  FJOURNAL = {Mathematical Research Letters},
    VOLUME = {18},
      YEAR = {2011},
    NUMBER = {5},
     PAGES = {991--1000},
      ISSN = {1073-2780},
   MRCLASS = {14E30},
  MRNUMBER = {2875871},
MRREVIEWER = {Mihnea\ Popa},
       DOI = {10.4310/MRL.2011.v18.n5.a16},
       URL = {https://doi.org/10.4310/MRL.2011.v18.n5.a16},
}

@article {GKP16-duke,
    AUTHOR = {Greb, Daniel and Kebekus, Stefan and Peternell, Thomas},
     TITLE = {\'{E}tale fundamental groups of {K}awamata log terminal spaces,
              flat sheaves, and quotients of abelian varieties},
   JOURNAL = {Duke Math. J.},
  FJOURNAL = {Duke Mathematical Journal},
    VOLUME = {165},
      YEAR = {2016},
    NUMBER = {10},
     PAGES = {1965--2004},
      ISSN = {0012-7094},
   MRCLASS = {14B25 (14B05 14E30 14J17)},
  MRNUMBER = {3522654},
MRREVIEWER = {James McKernan},
       DOI = {10.1215/00127094-3450859},
       URL = {https://doi.org/10.1215/00127094-3450859},
}

@article {GJZZ08,
    AUTHOR = {Grunewald, F. and Jaikin-Zapirain, A. and Zalesskii, P. A.},
     TITLE = {Cohomological goodness and the profinite completion of
              {B}ianchi groups},
   JOURNAL = {Duke Math. J.},
  FJOURNAL = {Duke Mathematical Journal},
    VOLUME = {144},
      YEAR = {2008},
    NUMBER = {1},
     PAGES = {53--72},
      ISSN = {0012-7094,1547-7398},
   MRCLASS = {20E18 (11F06 19B37 20H05 20H10 20J05)},
  MRNUMBER = {2429321},
MRREVIEWER = {John\ G.\ Ratcliffe},
       DOI = {10.1215/00127094-2008-031},
       URL = {https://doi.org/10.1215/00127094-2008-031},
}

@incollection {HPS18,
    AUTHOR = {Hacon, Christopher and Popa, Mihnea and Schnell, Christian},
     TITLE = {Algebraic fiber spaces over abelian varieties: around a recent
              theorem by {C}ao and {P}\u{a}un},
 BOOKTITLE = {Local and global methods in algebraic geometry},
    SERIES = {Contemp. Math.},
    VOLUME = {712},
     PAGES = {143--195},
 PUBLISHER = {Amer. Math. Soc., [Providence], RI},
      YEAR = {2018},
   MRCLASS = {14E30 (14K05)},
  MRNUMBER = {3832403},
MRREVIEWER = {Kenta Hashizume},
       DOI = {10.1090/conm/712/14346},
       URL = {https://doi.org/10.1090/conm/712/14346},
}

@article {HM07,
    AUTHOR = {Hacon, Christopher D. and Mckernan, James},
     TITLE = {On {S}hokurov's rational connectedness conjecture},
   JOURNAL = {Duke Math. J.},
  FJOURNAL = {Duke Mathematical Journal},
    VOLUME = {138},
      YEAR = {2007},
    NUMBER = {1},
     PAGES = {119--136},
      ISSN = {0012-7094},
   MRCLASS = {14E30 (14E05 14J45)},
  MRNUMBER = {2309156},
MRREVIEWER = {Mihnea Popa},
       DOI = {10.1215/S0012-7094-07-13813-4},
       URL = {https://doi.org/10.1215/S0012-7094-07-13813-4},
}

@misc {HX26,
    AUTHOR = {Hacon, Christopher and Xie, Lingyao},
     TITLE = {On the {K}\"ahler {MMP} and the transcendental
              base-point-free theorem},
      YEAR = {2026},
      NOTE = {Preprint, arXiv:2607.24986},
    EPRINT = {2607.24986},
ARCHIVEPREFIX = {arXiv},
PRIMARYCLASS = {math.AG},
       URL = {https://arxiv.org/abs/2607.24986},
}

@article {HIM22,
    AUTHOR = {Hosono, Genki and Iwai, Masataka and Matsumura, Shin-ichi},
     TITLE = {On projective manifolds with pseudo-effective tangent bundle},
   JOURNAL = {J. Inst. Math. Jussieu},
  FJOURNAL = {Journal of the Institute of Mathematics of Jussieu. JIMJ.
              Journal de l'Institut de Math\'{e}matiques de Jussieu},
    VOLUME = {21},
      YEAR = {2022},
    NUMBER = {5},
     PAGES = {1801--1830},
      ISSN = {1474-7480,1475-3030},
   MRCLASS = {32J25 (14J26 32L99 58A30)},
  MRNUMBER = {4476130},
MRREVIEWER = {Ahmed\ Lesfari},
       DOI = {10.1017/S1474748020000754},
       URL = {https://doi.org/10.1017/S1474748020000754},
}

@article {HLS22,
    AUTHOR = {H\"{o}ring, Andreas and Liu, Jie and Shao, Feng},
     TITLE = {Examples of {F}ano manifolds with non-pseudoeffective tangent
              bundle},
   JOURNAL = {J. Lond. Math. Soc. (2)},
  FJOURNAL = {Journal of the London Mathematical Society. Second Series},
    VOLUME = {106},
      YEAR = {2022},
    NUMBER = {1},
     PAGES = {27--59},
      ISSN = {0024-6107},
   MRCLASS = {14J45 (14E30 14J40)},
  MRNUMBER = {4454484},
MRREVIEWER = {Enrica Floris},
       DOI = {10.1112/jlms.12567},
       URL = {https://doi.org/10.1112/jlms.12567},
}

@article {HP19,
    AUTHOR = {H\"{o}ring, Andreas and Peternell, Thomas},
     TITLE = {Algebraic integrability of foliations with numerically trivial
              canonical bundle},
   JOURNAL = {Invent. Math.},
  FJOURNAL = {Inventiones Mathematicae},
    VOLUME = {216},
      YEAR = {2019},
    NUMBER = {2},
     PAGES = {395--419},
      ISSN = {0020-9910},
   MRCLASS = {14J30 (14J32 37F75)},
  MRNUMBER = {3953506},
MRREVIEWER = {James McKernan},
       DOI = {10.1007/s00222-018-00853-2},
       URL = {https://doi.org/10.1007/s00222-018-00853-2},
}

@misc {IMZ23,
    AUTHOR = {Iwai, Masataka and Matsumura, Shin-ichi and Zhong, Guolei},
     TITLE = {Positivity of tangent sheaves of projective varieties---the
              structure of {MRC} fibrations},
      YEAR = {2023},
      NOTE = {To appear in Algebr. Geom., arXiv:2309.09489},
    EPRINT = {2309.09489},
ARCHIVEPREFIX = {arXiv},
PRIMARYCLASS = {math.AG},
       URL = {https://arxiv.org/abs/2309.09489},
}

@misc {IJZ25,
    AUTHOR = {Iwai, Masataka and Jinnouchi, Satoshi and Zhang, Shiyu},
     TITLE = {The {M}iyaoka--{Y}au inequality for singular varieties with
              big canonical or anticanonical divisors},
      YEAR = {2025},
      NOTE = {Preprint, arXiv:2507.08522},
    EPRINT = {2507.08522},
ARCHIVEPREFIX = {arXiv},
PRIMARYCLASS = {math.AG},
       URL = {https://arxiv.org/abs/2507.08522},
}

@article {JLZ25,
    AUTHOR = {Jia, Jia and Lee, Yongnam and Zhong, Guolei},
     TITLE = {Smooth projective surfaces with pseudo-effective tangent
              bundles},
   JOURNAL = {J. Math. Soc. Japan},
  FJOURNAL = {Journal of the Mathematical Society of Japan},
    VOLUME = {77},
      YEAR = {2025},
    NUMBER = {1},
     PAGES = {75--102},
      ISSN = {0025-5645,1881-1167},
   MRCLASS = {14J60 (14E30 14J26 14J27)},
  MRNUMBER = {4854769},
MRREVIEWER = {Damian\ M.\ Maingi},
       DOI = {10.2969/jmsj/91579157},
       URL = {https://doi.org/10.2969/jmsj/91579157},
}

@article {KT24,
    AUTHOR = {Kawakami, Tatsuro and Totaro, Burt},
     TITLE = {Endomorphisms of varieties and {B}ott vanishing},
   JOURNAL = {J. Algebraic Geom.},
  FJOURNAL = {Journal of Algebraic Geometry},
    VOLUME = {34},
      YEAR = {2025},
    NUMBER = {2},
     PAGES = {381--405},
      ISSN = {1056-3911,1534-7486},
  MRCLASS = {14E05 (13A35 14G17 14J45)},
  MRNUMBER = {4876293},
MRREVIEWER = {Haidong\ Liu},
       DOI = {10.1090/jag/838},
       URL = {https://doi.org/10.1090/jag/838},
}

@article {KS21,
    AUTHOR = {Kebekus, Stefan and Schnell, Christian},
     TITLE = {Extending holomorphic forms from the regular locus of a
              complex space to a resolution of singularities},
   JOURNAL = {J. Amer. Math. Soc.},
  FJOURNAL = {Journal of the American Mathematical Society},
    VOLUME = {34},
      YEAR = {2021},
    NUMBER = {2},
     PAGES = {315--368},
      ISSN = {0894-0347,1088-6834},
   MRCLASS = {32S20 (14B05 14B15)},
  MRNUMBER = {4280862},
MRREVIEWER = {Dmitry\ Kerner},
       DOI = {10.1090/jams/962},
       URL = {https://doi.org/10.1090/jams/962},
}

@article {Kol93,
    AUTHOR = {Koll\'{a}r, J\'{a}nos},
     TITLE = {Shafarevich maps and plurigenera of algebraic varieties},
   JOURNAL = {Invent. Math.},
  FJOURNAL = {Inventiones Mathematicae},
    VOLUME = {113},
      YEAR = {1993},
    NUMBER = {1},
     PAGES = {177--215},
      ISSN = {0020-9910},
   MRCLASS = {14E20 (14E30 14J10)},
  MRNUMBER = {1223229},
MRREVIEWER = {Alessio Corti},
       DOI = {10.1007/BF01244307},
       URL = {https://doi.org/10.1007/BF01244307},
}

@article {Mat23,
    AUTHOR = {Matsumura, Shin-ichi},
     TITLE = {On the minimal model program for projective varieties with
              pseudo-effective tangent sheaf},
   JOURNAL = {\'Epijournal G\'eom. Alg\'ebrique},
  FJOURNAL = {\'Epijournal de G\'eom\'etrie Alg\'ebrique. EPIGA},
    VOLUME = {7},
      YEAR = {2023},
     PAGES = {Art. 22, 13},
      ISSN = {2491-6765},
   MRCLASS = {14E30 (32J25 53C25)},
  MRNUMBER = {4671729},
       DOI = {10.46298/epiga.2023.10337},
       URL = {https://doi.org/10.46298/epiga.2023.10337},
}

@article {MQ25,
    AUTHOR = {Matsumura, Shin-ichi and Qing, Chenghao},
     TITLE = {On compact {K}\"ahler manifolds with pseudo-effective tangent
              bundle},
   JOURNAL = {Int. Math. Res. Not. IMRN},
  FJOURNAL = {International Mathematics Research Notices. IMRN},
      YEAR = {2026},
    NUMBER = {10},
     PAGES = {Paper No. rnag096},
      ISSN = {1073-7928,1687-0247},
   MRCLASS = {32Q15},
  MRNUMBER = {5073802},
       DOI = {10.1093/imrn/rnag096},
       URL = {https://doi.org/10.1093/imrn/rnag096},
}

@article {MY21,
    AUTHOR = {Matsuzawa, Yohsuke and Yoshikawa, Shou},
     TITLE = {Int-amplified endomorphisms on normal projective surfaces},
   JOURNAL = {Taiwanese J. Math.},
  FJOURNAL = {Taiwanese Journal of Mathematics},
    VOLUME = {25},
      YEAR = {2021},
    NUMBER = {4},
     PAGES = {681--697},
      ISSN = {1027-5487,2224-6851},
   MRCLASS = {14E30},
  MRNUMBER = {4298918},
MRREVIEWER = {Tatiana\ M.\ Bandman},
       DOI = {10.11650/tjm/210101},
       URL = {https://doi.org/10.11650/tjm/210101},
}

@article {Men20,
    AUTHOR = {Meng, Sheng},
     TITLE = {Building blocks of amplified endomorphisms of normal
              projective varieties},
   JOURNAL = {Math. Z.},
  FJOURNAL = {Mathematische Zeitschrift},
    VOLUME = {294},
      YEAR = {2020},
    NUMBER = {3-4},
     PAGES = {1727--1747},
      ISSN = {0025-5874,1432-1823},
   MRCLASS = {14E30 (08A35 32H50)},
  MRNUMBER = {4074056},
MRREVIEWER = {Haidong\ Liu},
       DOI = {10.1007/s00209-019-02316-7},
       URL = {https://doi.org/10.1007/s00209-019-02316-7},
}

@article {MZ19,
    AUTHOR = {Meng, Sheng and Zhang, De-Qi},
     TITLE = {Characterizations of toric varieties via polarized
              endomorphisms},
   JOURNAL = {Math. Z.},
  FJOURNAL = {Mathematische Zeitschrift},
    VOLUME = {292},
      YEAR = {2019},
    NUMBER = {3-4},
     PAGES = {1223--1231},
      ISSN = {0025-5874,1432-1823},
   MRCLASS = {14M25 (20K30 32H50)},
  MRNUMBER = {3980290},
MRREVIEWER = {Ozhan\ Genc},
       DOI = {10.1007/s00209-018-2160-8},
       URL = {https://doi.org/10.1007/s00209-018-2160-8},
}

@article {MZ22,
    AUTHOR = {Meng, Sheng and Zhang, De-Qi},
     TITLE = {Kawaguchi-{S}ilverman conjecture for certain surjective
              endomorphisms},
   JOURNAL = {Doc. Math.},
  FJOURNAL = {Documenta Mathematica},
    VOLUME = {27},
      YEAR = {2022},
     PAGES = {1605--1642},
      ISSN = {1431-0635,1431-0643},
   MRCLASS = {37P55 (08A35 14E05 14E30)},
  MRNUMBER = {4574221},
       DOI = {10.4171/dm/x13},
       URL = {https://doi.org/10.4171/dm/x13}}

@incollection {MZ23-survey,
    AUTHOR = {Meng, Sheng and Zhang, De-Qi},
     TITLE = {Advances in the {E}quivariant {M}inimal {M}odel {P}rogram and
              {T}heir {A}pplications in {C}omplex and {A}rithmetic
              {D}ynamics},
 BOOKTITLE = {Algebraic, complex, and arithmetic dynamics},
    SERIES = {Simons Symp.},
     PAGES = {99--123},
 PUBLISHER = {Springer, Cham},
      YEAR = {[2026] \copyright 2026},
      ISBN = {978-3-032-04047-3; 978-3-032-04048-0},
   MRCLASS = {99-06},
  MRNUMBER = {5098186},
       DOI = {10.1007/978-3-032-04048-0\_4},
       URL = {https://doi.org/10.1007/978-3-032-04048-0_4},
}

@misc {MZ-TJM,
    AUTHOR = {Matsumura, Shin-ichi and Zhong, Guolei},
     TITLE = {A note on compact K\"ahler varieties with pseudo-effective tangent sheaf},
      YEAR = {2026},
        NOTE = {Preprint, arXiv:2609.07069 },
    EPRINT = {2609.07069 },
ARCHIVEPREFIX = {arXiv},
PRIMARYCLASS = {math.AG},
       URL = {https://arxiv.org/abs/2609.07069 },
}

@misc {MYY24-lcy,
    AUTHOR = {Moraga, Joaqu{\'\i}n and Y{\'a}{\~n}ez, Jos{\'e} Ignacio and
              Yeong, Wern},
     TITLE = {Polarized endomorphisms of log {C}alabi--{Y}au pairs},
      YEAR = {2024},
      NOTE = {Preprint, arXiv:2406.18092},
    EPRINT = {2406.18092},
ARCHIVEPREFIX = {arXiv},
PRIMARYCLASS = {math.AG},
       URL = {https://arxiv.org/abs/2406.18092},
}

@article {Nak02,
    AUTHOR = {Nakayama, Noboru},
     TITLE = {Ruled surfaces with non-trivial surjective endomorphisms},
   JOURNAL = {Kyushu J. Math.},
  FJOURNAL = {Kyushu Journal of Mathematics},
    VOLUME = {56},
      YEAR = {2002},
    NUMBER = {2},
     PAGES = {433--446},
      ISSN = {1340-6116,1883-2032},
   MRCLASS = {14J26},
  MRNUMBER = {1934136},
MRREVIEWER = {Sandra\ Di Rocco},
       DOI = {10.2206/kyushujm.56.433},
       URL = {https://doi.org/10.2206/kyushujm.56.433},
}

@article {NZ10,
    AUTHOR = {Nakayama, Noboru and Zhang, De-Qi},
     TITLE = {Polarized endomorphisms of complex normal varieties},
   JOURNAL = {Math. Ann.},
  FJOURNAL = {Mathematische Annalen},
    VOLUME = {346},
      YEAR = {2010},
    NUMBER = {4},
     PAGES = {991--1018},
      ISSN = {0025-5831,1432-1807},
   MRCLASS = {14J10 (14H20 32H50)},
  MRNUMBER = {2587100},
MRREVIEWER = {Karl\ Schwede},
       DOI = {10.1007/s00208-009-0420-y},
       URL = {https://doi.org/10.1007/s00208-009-0420-y},
}

@misc {Ou22,
    AUTHOR = {Ou, Wenhao},
     TITLE = {Admissible metrics on compact {K}\"ahler varieties},
      YEAR = {2022},
      NOTE = {Preprint, arXiv:2201.04821},
    EPRINT = {2201.04821},
ARCHIVEPREFIX = {arXiv},
PRIMARYCLASS = {math.AG},
       URL = {https://arxiv.org/abs/2201.04821},
}

@misc {Ou25,
    AUTHOR = {Ou, Wenhao},
     TITLE = {A characterization of uniruled compact {K}\"ahler manifolds},
      YEAR = {2025},
      NOTE = {Preprint, arXiv:2501.18088},
    EPRINT = {2501.18088},
ARCHIVEPREFIX = {arXiv},
PRIMARYCLASS = {math.AG},
       URL = {https://arxiv.org/abs/2501.18088},
}

@misc {FO25,
    AUTHOR = {Fu, Xin and Ou, Wenhao},
     TITLE = {Orbifold {B}ogomolov--{G}ieseker inequalities on compact
              {K}\"ahler varieties},
      YEAR = {2025},
      NOTE = {Preprint, arXiv:2511.03530},
    EPRINT = {2511.03530},
ARCHIVEPREFIX = {arXiv},
PRIMARYCLASS = {math.DG},
       URL = {https://arxiv.org/abs/2511.03530},
}

@article {PS89,
    AUTHOR = {Paranjape, K. H. and Srinivas, V.},
     TITLE = {Self-maps of homogeneous spaces},
   JOURNAL = {Invent. Math.},
  FJOURNAL = {Inventiones Mathematicae},
    VOLUME = {98},
      YEAR = {1989},
    NUMBER = {2},
     PAGES = {425--444},
      ISSN = {0020-9910,1432-1297},
   MRCLASS = {14M17 (20G20)},
  MRNUMBER = {1016272},
MRREVIEWER = {Vladimir\ L.\ Popov},
       DOI = {10.1007/BF01388861},
       URL = {https://doi.org/10.1007/BF01388861},
}

@article {PT18,
    AUTHOR = {P\u{a}un, Mihai and Takayama, Shigeharu},
     TITLE = {Positivity of twisted relative pluricanonical bundles and their direct images},
   JOURNAL = {J. Algebraic Geom.},
  FJOURNAL = {Journal of Algebraic Geometry},
    VOLUME = {27},
      YEAR = {2018},
    NUMBER = {2},
     PAGES = {211--272},
      ISSN = {1056-3911},
   MRCLASS = {32L05 (32A25)},
  MRNUMBER = {3764276},
MRREVIEWER = {Severin Barmeier},
       DOI = {10.1090/jag/702},
       URL = {https://doi.org/10.1090/jag/702},
}

@book {Ser97,
    AUTHOR = {Serre, Jean-Pierre},
     TITLE = {Galois cohomology},
      NOTE = {Translated from the French by Patrick Ion and revised by the
              author},
 PUBLISHER = {Springer-Verlag, Berlin},
      YEAR = {1997},
     PAGES = {x+210},
      ISBN = {3-540-61990-9},
   MRCLASS = {12G05 (11R34)},
  MRNUMBER = {1466966},
       DOI = {10.1007/978-3-642-59141-9},
       URL = {https://doi.org/10.1007/978-3-642-59141-9},
}

@article {SZ25,
    AUTHOR = {Shao, Feng and Zhong, Guolei},
     TITLE = {Bigness of tangent bundles and dynamical rigidity of {F}ano
              manifolds of {P}icard number 1 (with an appendix by {J}ie
              {L}iu)},
   JOURNAL = {Math. Ann.},
  FJOURNAL = {Mathematische Annalen},
    VOLUME = {391},
      YEAR = {2025},
    NUMBER = {2},
     PAGES = {1731--1752},
      ISSN = {0025-5831,1432-1807},
   MRCLASS = {14J40 (14J45)},
  MRNUMBER = {4853004},
       DOI = {10.1007/s00208-024-02955-0},
       URL = {https://doi.org/10.1007/s00208-024-02955-0},
}

@article {Shi25,
    AUTHOR = {Shibata, Takahiro},
     TITLE = {Q-abelian and  {$\mathbb{Q}$}-{F}ano finite quotients of abelian
              varieties},
   JOURNAL = {Kyushu J. Math.},
  FJOURNAL = {Kyushu Journal of Mathematics},
    VOLUME = {79},
      YEAR = {2025},
    NUMBER = {2},
     PAGES = {397--412},
      ISSN = {1340-6116,1883-2032},
   MRCLASS = {14},
  MRNUMBER = {5078215},
}

@article {Tak03,
    AUTHOR = {Takayama, Shigeharu},
     TITLE = {Local simple connectedness of resolutions of log-terminal
              singularities},
   JOURNAL = {Internat. J. Math.},
  FJOURNAL = {International Journal of Mathematics},
    VOLUME = {14},
      YEAR = {2003},
    NUMBER = {8},
     PAGES = {825--836},
      ISSN = {0129-167X,1793-6519},
   MRCLASS = {14E30 (14B05 14E20 32J25)},
  MRNUMBER = {2013147},
MRREVIEWER = {Massimiliano\ Mella},
       DOI = {10.1142/S0129167X0300196X},
       URL = {https://doi.org/10.1142/S0129167X0300196X},
}

@article {Wang21,
    AUTHOR = {Wang, Juanyong},
     TITLE = {On the {I}itaka conjecture {$C_{n,m}$} for {K}\"{a}hler fibre
              spaces},
   JOURNAL = {Ann. Fac. Sci. Toulouse Math. (6)},
  FJOURNAL = {Annales de la Facult\'{e} des Sciences de Toulouse. Math\'{e}matiques.
              S\'{e}rie 6},
    VOLUME = {30},
      YEAR = {2021},
    NUMBER = {4},
     PAGES = {813--897},
      ISSN = {0240-2963},
   MRCLASS = {32Q15 (14E30 32C15 53C55)},
  MRNUMBER = {4350100},
MRREVIEWER = {Rare\c{s} R\u{a}sdeaconu},
       DOI = {10.5802/afst.1690},
       URL = {https://doi.org/10.5802/afst.1690},
}

@misc {Xia26,
    AUTHOR = {Xia, Mingchen},
     TITLE = {Transcendental {$b$}-divisors {II}---monotonicity theorem},
      YEAR = {2026},
      NOTE = {Preprint, arXiv:2603.14362},
    EPRINT = {2603.14362},
ARCHIVEPREFIX = {arXiv},
PRIMARYCLASS = {math.AG},
       URL = {https://arxiv.org/abs/2603.14362},
}

@article {Xu14,
    AUTHOR = {Xu, Chenyang},
     TITLE = {Finiteness of algebraic fundamental groups},
   JOURNAL = {Compos. Math.},
  FJOURNAL = {Compositio Mathematica},
    VOLUME = {150},
      YEAR = {2014},
    NUMBER = {3},
     PAGES = {409--414},
      ISSN = {0010-437X,1570-5846},
   MRCLASS = {14J17 (14J45)},
  MRNUMBER = {3187625},
MRREVIEWER = {Andrzej\ Kozlowski},
       DOI = {10.1112/S0010437X13007562},
       URL = {https://doi.org/10.1112/S0010437X13007562},
}

@article {Yos21,
    AUTHOR = {Yoshikawa, Shou},
     TITLE = {Structure of {F}ano fibrations of varieties admitting an
              int-amplified endomorphism},
   JOURNAL = {Adv. Math.},
  FJOURNAL = {Advances in Mathematics},
    VOLUME = {391},
      YEAR = {2021},
     PAGES = {Paper No. 107964, 32},
      ISSN = {0001-8708,1090-2082},
   MRCLASS = {14E30 (14D06 14J45 37F99)},
  MRNUMBER = {4305238},
MRREVIEWER = {Sho\ Tanimoto},
       DOI = {10.1016/j.aim.2021.107964},
       URL = {https://doi.org/10.1016/j.aim.2021.107964},
}

@incollection {ZhaSW06,
    AUTHOR = {Zhang, Shou-Wu},
     TITLE = {Distributions in algebraic dynamics},
 BOOKTITLE = {Surveys in differential geometry. {V}ol. {X}},
    SERIES = {Surv. Differ. Geom.},
    VOLUME = {10},
     PAGES = {381--430},
 PUBLISHER = {Int. Press, Somerville, MA},
      YEAR = {2006},
      ISBN = {978-1-57146-116-2; 1-57146-116-7},
   MRCLASS = {32H50 (11G50 32J27 37F10)},
  MRNUMBER = {2408228},
MRREVIEWER = {Matthew\ H.\ Baker},
       DOI = {10.4310/SDG.2005.v10.n1.a9},
       URL = {https://doi.org/10.4310/SDG.2005.v10.n1.a9},
}

@article {Zha14,
    AUTHOR = {Zhang, De-Qi},
     TITLE = {Invariant hypersurfaces of endomorphisms of projective
              varieties},
   JOURNAL = {Adv. Math.},
  FJOURNAL = {Advances in Mathematics},
    VOLUME = {252},
      YEAR = {2014},
     PAGES = {185--203},
      ISSN = {0001-8708,1090-2082},
   MRCLASS = {14B05 (14C20 14M22 32H50 37F10)},
  MRNUMBER = {3144228},
MRREVIEWER = {Mattias\ Jonsson},
       DOI = {10.1016/j.aim.2013.10.014},
       URL = {https://doi.org/10.1016/j.aim.2013.10.014},
}

@article {Zho21-Asian,
    AUTHOR = {Zhong, Guolei},
     TITLE = {Int-amplified endomorphisms of compact {K}\"ahler spaces},
   JOURNAL = {Asian J. Math.},
  FJOURNAL = {Asian Journal of Mathematics},
    VOLUME = {25},
      YEAR = {2021},
    NUMBER = {3},
     PAGES = {369--392},
      ISSN = {1093-6106,1945-0036},
   MRCLASS = {14E30 (14J50)},
  MRNUMBER = {4395605},
MRREVIEWER = {Chen\ Jiang},
       DOI = {10.4310/AJM.2021.v25.n3.a3},
       URL = {https://doi.org/10.4310/AJM.2021.v25.n3.a3},
}

@article {Zho23,
    AUTHOR = {Zhong, Guolei},
     TITLE = {Existence of the equivariant minimal model program for compact
              {K}\"ahler threefolds with the action of an abelian group of
              maximal rank},
   JOURNAL = {Math. Nachr.},
  FJOURNAL = {Mathematische Nachrichten},
    VOLUME = {296},
      YEAR = {2023},
    NUMBER = {7},
     PAGES = {3128--3135},
      ISSN = {0025-584X,1522-2616},
   MRCLASS = {14E30 (14J30 14J50)},
  MRNUMBER = {4626875},
MRREVIEWER = {G.\ K.\ Sankaran},
       DOI = {10.1002/mana.202200127},
       URL = {https://doi.org/10.1002/mana.202200127},
}

\end{document}